\documentclass[12pt]{amsart}

\usepackage[matrix,arrow,curve,frame]{xy}
\usepackage{amsmath,amsthm,amssymb,enumerate}
\usepackage{latexsym}
\usepackage{amscd}
\usepackage[colorlinks=true]{hyperref}
\usepackage{euscript}

\newtheorem{thm}[subsection]{Theorem}
\newtheorem{defn}[subsection]{Definition}

\newtheorem{claim}[subsection]{Claim}

\newtheorem{lemma}[subsection]{Lemma}

\newtheorem{remark}[subsection]{Remark}

\newtheorem{conv}[subsection]{Convention}

\theoremstyle{definition}
\newtheorem{example}[subsection]{Example}
\newtheorem{question}[subsection]{Question}

\newcommand{\R}{\mathbb R}
\newcommand{\Q}{\mathbb Q}
\newcommand{\Z}{\mathbb Z}

\newcommand{\C}{\mathbb C}

\newcommand{\db}{\bar{\partial}}

\DeclareMathOperator{\Aut}{Aut}
\DeclareMathOperator{\colim}{colim}

\DeclareMathOperator{\Map}{Map}

\DeclareMathOperator{\Hom}{Hom}

\DeclareMathOperator{\Symp}{Symp}
\DeclareMathOperator{\BSymp}{BSymp}
\DeclareMathOperator{\ESymp}{ESymp}
\DeclareMathOperator{\BAut}{BAut}
\DeclareMathOperator{\BGL}{BGL}
\DeclareMathOperator{\GL}{GL}
\DeclareMathOperator{\BO}{BO}
\DeclareMathOperator{\BU}{BU}
\DeclareMathOperator{\BSpin}{BSpin}
\DeclareMathOperator{\BSO}{BSO}
\DeclareMathOperator{\BtU}{B\tilde{U}}
\DeclareMathOperator{\U}{U}
\DeclareMathOperator{\EU}{EU}

\DeclareMathOperator{\Spin}{Spin}
\DeclareMathOperator{\SO}{SO}

\DeclareMathOperator{\UO}{U/O}
\DeclareMathOperator{\USO}{U/SO}

\DeclareMathOperator{\tUSpin}{\tilde{U}/Spin}
\DeclareMathOperator{\tU}{\tilde{U}}
\DeclareMathOperator{\Or}{O}

\DeclareMathOperator{\B}{B\tilde{O}}

\DeclareMathOperator{\MO}{MO}
\DeclareMathOperator{\MSO}{MSO}
\DeclareMathOperator{\MSpin}{MSpin}

\DeclareMathOperator{\KZtt}{K(\mathbb{Z}/2,2)}
\DeclareMathOperator{\KZto}{K(\mathbb{Z}/2,1)}

\DeclareMathOperator{\BZf}{B(\mathbb{Z}/4)}
\DeclareMathOperator{\E}{E}
\DeclareMathOperator{\Fr}{E}

\DeclareMathOperator{\Fm}{F}

\DeclareMathOperator{\BSo}{BS^1}

\DeclareMathOperator{\Hrt}{H(\Z/2)}

\DeclareMathOperator{\Jb}{J}

\newfont{\german}{eufm10}

 \DeclareMathOperator{\End}{End}

\newcommand\qu{/\kern-.7ex/}

\begin{document}
\pagestyle{plain}

\title
{The Stable Symplectic Category and Quantization}
\author{Nitu Kitchloo}
\address{Department of Mathematics, Johns Hopkins University, Baltimore, USA}
\email{nitu@math.jhu.edu}
\thanks{Nitu Kitchloo is supported in part by NSF through grant DMS
  1005391.}

\date{\today}

\dedicatory{To Graeme, Gunnar, Ib, Ish and Ralph with best wishes for many more creative and healthy years.}

{\abstract

\noindent
We study a stabilization of the symplectic category introduced by A. Weinstein as a domain for the geometric quantization functor. The symplectic category is a topological category with objects given by symplectic manifolds, and morphisms being suitable lagrangian correspondences. The main drawback of Weinstein's symplectic category is that composition of morphisms cannot always be defined. Our stabilization procedure rectifies this problem while remaining faithful to the original notion of composition. The stable symplectic category is enriched over the category of spectra (in particular, its morphisms can be described as infinite loop spaces representing the space of immersed lagrangians), and it possesses several appealing properties that are relevant to deformation, and geometric quantization.}
\maketitle

\tableofcontents
\section{Introduction}

\noindent
Motivated by earlier work by Guillemin and Sternberg \cite{GS}, A. Weinstein \cite{W,W2} introduced the symplectic category as a domain category for constructing the (yet to be completely defined) geometric quantization functor. The variant of Weinstein's symplectic category we consider is a topological category with objects given by symplectic manifolds, and morphisms between two symplectic manifolds $(M,\omega)$ and $(N,\eta)$ are defined as lagrangian immersions to $\overline{M} \times N$, where the conjugate symplectic manifold $\overline{M}$ is defined by the pair $(M,-\omega)$. Geometric quantization is an attempt at constructing a canonical representation of this category. There are many categories closely related to the symplectic category (for example the Fukaya 2-category), that we shall not consider in this document. 

\medskip
\noindent
Composition of two lagrangians $L_1 \looparrowright \overline{M} \times N$ and $L_2 \looparrowright \overline{N} \times K$ in the symplectic category is defined as $L_1 \ast L_2$ given by the cartesian product $L_1 \times_N L_2 \longrightarrow \overline{M} \times N$. Alternatively, one may define $L_1 \ast L_2$ as the intersection of $L_1 \times L_2$ with $ \overline{M} \times \Delta(N) \times K$ inside $\overline{M} \times N \times \overline{N} \times K$, where $\Delta(N) \subset N \times \overline{N}$ is the diagonal submanifold. The first observation to make is that composition as we have defined above does not always give rise to a lagrangian immersion. For composition to yield a lagrangian immersion, the intersection that is used to define it must be transverse. In particular, the symplectic category fails to be a viable category.

\medskip
\noindent
One way to fix the problem with the failure of composition was introduced by Wehrheim and Woodward \cite{WW} where they consider the free category generated by all the correspondences, modulo the obvious relation if the pair of composable morphisms is transverse. Here, we describe another method of extending the symplectic category into an honest category. This document is a more detailed version of \cite{N}. We introduce a moduli space of stabilized lagrangian immersions in a symplectic manifold of the form $\overline{M} \times N$ \footnote{under the assumption of monotonicity. Otherwise, one has the space of totally real immersions.}. This moduli space can be described as the infinite loop space corresponding to a certain Thom spectrum. Taking this as the space of morphisms defines the {\em Stable Symplectic Category} that is naturally enriched over the monoidal category of spectra (under smash product). Composition in this stable symplectic category remains faithful to the original definition introduced by Weinstein. There are variants of the stable symplectic category known as the stable oriented and the stable metaplectic category.

\medskip
\noindent
Geometrically, stabilization of Weinstein's symplectic category can be seen as ``inverting the symplectic manifold $\C$". In other words, we introduce a relation on the symplectic category that identifies two symplectic manifolds $M$ and $N$ if $M \times \C^k$ becomes equivalent to $N \times \C^k$ for some $k$. This procedure of stabilization is motivated by applications we have in mind. More precisely, we are interested in exploring the existence of a ``derived geometric quantization functor" that takes values in a suitable category of virtual Hilbert spaces (for example Kasparov's Fredholm modules \cite{K}). Since the quantization of the manifold $\C$ is the unique irreducible representation of the Weyl algebra of differential operators on $\R$, it is reasonable to expect that the quantization of a symplectic manifold $M$ is equivalent in a derived sense to the quantization of $M \times \C$. We shall go deeper into this application in a forthcoming document \cite{N2}. In particular, we shall consider a natural extensions of the stable symplectic (or metaplectic) category from the standpoint of geometric quantization that involves extending coefficients in our category by a flat line bundle. We call this the {\em Stable Symplectic Category of Symbols}. As before, there will be variants called the stable oriented and stable metaplectic category of symbols. By linearizing the stable metaplectic category of symbols along the $\hat{A}$-genus and expressing the result in terms of Kasparov's bivariant K-theory, one may categorify geometric quantization. Details will appear in \cite{N2}. See section \ref{meta} for more discussion on this subject.

\medskip
\noindent
Having stabilized, we notice the appearance of structure relevant to the categorical aspects of quantization. To begin with, we observe that there is an intersection pairing between stable lagrangians inside a symplectic manifold and those inside its conjugate (see \ref{D}). This allows us to construct algebraic representations of the stable symplectic category (see \ref{rep}). Endomorphisms of an object $(M,\omega)$ in our category can be seen as a homotopical notion of the ``algebra of observables". Indeed, in section \ref{comp} we show that this algebra is an $A_\infty$-deformation of an $E_\infty$-algebra. In section \ref{stab} we also construct a canonical representation of the symplectomorphism group of a compact symplectic manifold $M$ in this algebra.

\medskip
\noindent
Another direction that we will pursue in a later document \cite{KM} is the question of the {\em Motivic Galois Group} of the stable symplectic category. We begin with the observation that there exists a canonical monoidal functor from the stable symplectic category into the category of modules over a certain ``coefficient spectrum'' $\Omega$ that is naturally associated to the stable symplectic category (see section \ref{hcat}). By extending coefficients to other algebras over $\Omega$, one has a family of such functors, and one may ask for the structure of the Motivic Galois group of monoidal automorphisms of this family. In \cite{KM} we answer this question, and draw a parallel between the Motivic Galois group and the Gothendieck Teichm\"uller group \cite{KO} (also see question \ref{GT}). 

\medskip
\noindent
This document is organized as follows: In sections \ref{two}, \ref{hcat} we construct a category of stabilized symplectic manifolds enriched over the homotopy category of spectra. We call this the {\em Stable Symplectic Homotopy Category}. These sections are intended to establish a transparent connection between geometry, and the homotopical objects we use to represent it. The next two sections: \ref{four} and \ref{stab} go deeper into the structure of the stable symplectic homotopy category. Later sections \ref{A-mor} and \ref{comp} aim to make the symplecitic category rigid by using the (arguably opaque) language of parametrized $\mbox{S}$-modules. These sections allows us to give the stable symplectic category the foundation of an $A_\infty$-category enriched over the honest category of structured spectra.  The reader unfamiliar with the language of $A_\infty$-categories or structured spectra may wish to ignore those sections. 

\smallskip
\noindent
Before we begin, we would like to thank Gustavo Granja for his interest in this project and his hospitality at the IST (Lisbon) where this project started. We also thank Jack Morava for his continued interest and encouragement, and for sharing \cite{M} with us. In addition, we would like to thank David Ayala, John Klein, John Lind and Alan Weinstein for helpful conversations related to various parts of this project. And finally, we would like to acknowledge our debt to Peter Landweber for carefully reading an earlier version of this manuscript and providing several very helpful suggestions.

\section{Stabilized lagrangian immersions} \label{two}

\noindent
In this section we will describe the stabilization procedure that we will apply to the symplectic category in later sections. To begin with consider a symplectic manifold $(M^{2m},\omega)$. We fix a compatible almost complex structure $J$ on $M$. This endows the tangent bundle $\tau$ of $M$ with a unitary structure. Given an injection of unitary bundles $j : TM \longrightarrow M \times \C^\infty$, taking values in some $M \times \C^k$ for some large $k$, the complex Gauss map for $j$ yields a canonical map $\tau : M \longrightarrow \BU(m)$ that classifies the complex tangent bundle of $M$, and where our model of $\BU(m)$ is given by all complex $m$-planes in $\C^{\infty}$. 

\smallskip
\noindent
To motivate our constructions, let us also assume for the moment that the cohomology class of $\omega$ is a scalar multiple of the first Chern class of $M$. This assumption gives a geometric context to our construction, though it is not technically necessary (see remark \ref{totreal} for an explanation). 

\smallskip
\noindent
Given our model for $\BU(m)$ as the space of $m$-dimensional complex planes in $\C^{\infty}$, the universal space $\EU(m)$ can be identified with the space of all orthonormal complex $m$-frames in $\C^{\infty}$. We will choose $\EU(m)/\Or(m)$ as our model for $\BO(m)$. So we have a bundle $\BO(m) \longrightarrow \BU(m)$ with fiber $\U(m)/\Or(m)$. Consider the pullback diagram:
\[
\xymatrix{
{\mathcal{G}}(\tau)       \ar[d] \ar[r]^{\zeta \quad} & \BO(m) \ar[d] \\
M      \ar[r]^{\tau \quad} & \BU(m).
}
\]
Notice that the space ${\mathcal{G}}(\tau)$ has an intrinsic description as the bundle of lagrangian grassmannians on the tangent bundle of $M$. 
\begin{defn}\label{str}
Now let $X^m$ be an arbitrary $m$-manifold, and let $\zeta$ be an $m$-dimensional real vector bundle over a space $B$. By a $\zeta$-structure on $X$ we shall mean a bundle map $\tau(X) \longrightarrow \zeta$, where $\tau(X)$ is the tangent bundle of $X$.
\end{defn}

\begin{claim} \label{moduli}
The space of lagrangian immersions of $X$ into $M$ is homotopy equivalent to the space of $\zeta$-structures on the tangent bundle of $X$, where $\zeta$ is the vector bundle over ${\mathcal{G}}(\tau)$ defined by the above pullback.  
\end{claim}
\begin{proof}
A $\zeta$-structure on $\tau(X)$ is the same thing as a map of $X$ to $M$, along with an inclusion of $\tau(X)$ as an orthogonal (lagrangian) sub-bundle inside the pullback of the tangent bundle of $M$. Now the pullback of $[\omega]$ to $H^2(X,\R)$ factors through $H^2(\BO(m),\R)$, by our assumption that $[\omega] = c_1$ up to a scalar. Since $H^2(\BO(m), \R) = 0$, the h-principle \cite{EM} may now be invoked to show that the space of maps described above is equivalent to the space of lagrangian immersions of $X$ in $M$. 
\end{proof}

\noindent
Motivated by \cite{GMTW}, we consider the Thom spectrum ${\mathcal{G}}(\tau)^{-\zeta}$. An explicit model of this spectrum is obtained as follows. Let $\eta$ be given by the complement of $\zeta$ in $\C^k$. Let ${\mathcal{G}}(\tau)^\eta$ denote the Thom space of this bundle given by the identification space obtained from the bundle $\eta$ by compactifying all vectors at infinity. Define ${\mathcal{G}}(\tau)^{-\zeta}$ to be the spectrum $\Sigma^{-2k}{\mathcal{G}}(\tau)^\eta$, where $\Sigma^{-2k}$ denotes desuspension by the compactification of the vector space $\C^k$. 

\medskip
\noindent
The spectrum ${\mathcal{G}}(\tau)^{-\zeta}$ has the virtue of being a receptacle for compact immersed lagrangians in $M$. Indeed, for a compact manifold $X$ endowed with a lagrangian immersion into $M$, a stable map $[X] : \mbox{S} \longrightarrow {\mathcal{G}}(\tau)^{-\zeta}$ can be constructed as: $ [X] : \mbox{S} \longrightarrow X^{-\tau(X)} \longrightarrow {\mathcal{G}}(\tau)^{-\zeta}$, where $\mbox{S}$ denotes the sphere spetrum, and the first map $\mbox{S} \longrightarrow X^{-\tau(X)}$ is the Pontrjagin--Thom collapse map for some choice of embedding of $X$ in $\R^{\infty}$, and the second map is given by the (negative of the) $\zeta$-structure on $X$\footnote{Notice that the group of reparametrizations of the stable tangent bundle of $X$ acts on $X^{-\tau(X)}$, and may potentially change the immersion class of $X$. We thank Thomas Kargh for pointing this out.}. 

\medskip
\noindent
Since stable maps from $\mbox{S}$ to any spectrum $E$, represents points in the underlying infinite loop space $\Omega^{\infty} (E)$ of that spectrum, the above observation allows us to identify compact immersed lagrangians in $M$ with points in $\Omega^{\infty} ({\mathcal{G}}(\tau)^{-\zeta})$. To make this identification into an equivalence, we require a stabilization of the objects involved. We now proceed to describe the homotopical stabilization process in more detail. The geometric meaning of lagrangian stabilization is described following the description of homotopical stabilization. 

\smallskip
\noindent
Let $(M, \omega)$ be as above and let $\C^n$ be the the complex plane with its standard Hermitian structure. The bundle $\tau \oplus \C^n$ is the restriction of the tangent bundle of $M \times \C^n$ along $M \times \{ 0 \} \subseteq M \times \C^n$. Let ${\mathcal{G}}(\tau \oplus \C^n)$ be defined as the pullback: 
\[
\xymatrix{
{\mathcal{G}}(\tau \oplus \C^n)       \ar[d] \ar[r]^{\zeta_n} & \BO(m+n) \ar[d] \\
M     \ar[r]^{\tau \oplus \C^n \quad} & \BU(m+n).
}
\]
Notice that there is a canonical map:
\[ {\mathcal{G}}(\tau) \longrightarrow {\mathcal{G}}(\tau \oplus \C^n), \]
given by taking the fiberwise cartesian product of lagrangian planes in ${\mathcal{G}}(\tau)$ with the constant lagrangian subspace $\R^n$. Furthermore, the bundle $\zeta_n$ restricts to the bundle $\zeta \oplus \R^n$ along this map. To construct the spectrum ${\mathcal{G}}(\tau \oplus \C^n)^{-\zeta_n}$, we pick the embedding of $\tau \oplus \C^n$ into the trivial bundle $\C^{k+n}$ given by linearly extending the embedding of $\tau$ and proceed as before. 

\smallskip
\noindent
Notice that one has a natural map:
\[   \varphi_n : {\mathcal{G}}(\tau)^{-\zeta} \longrightarrow \Sigma^n \, {\mathcal{G}}(\tau \oplus \C^n)^{-\zeta_n}. \]
In general, given the standard inclusion $\R^{n_1} \subseteq \R^{n_2}$, we obtain a compatible family of natural maps representing a directed system which we call lagrangian stabilization: 
\[ \varphi_{n_1,n_2} : \Sigma^{n_1} {\mathcal{G}}(\tau \oplus \C^{n_1})^{-\zeta_{n_1}} \longrightarrow \Sigma^{n_2} {\mathcal{G}}(\tau \oplus \C^{n_2})^{-\zeta_{n_2}}. \] 
The individual spectra, and the directed system we just constructed have a geometric meaning. We say a few words about that in the following paragraph:

\bigskip
\noindent
{\bf The geometric meaning of lagrangian stabilization:}

\smallskip
\noindent
Let us briefly describe the geometric objects that our stabilization procedure captures. For $n > 0$, consider the Thom spectrum $\Sigma^n {\mathcal{G}}(\tau \oplus \C^n)^{-\zeta_n}$, where the notation is borrowed from the earlier part of this section. Now the methods described in \cite{A} (Sec. 4.4, 5.1), allows one to interpret the infinite loop space $\Omega^{\infty-n} ({\mathcal{G}}(\tau \oplus \C^n)^{-\zeta_n})$ as the moduli space of manifolds $L^{m+n} \subset \R^{\infty} \times \R^n$, with a proper projection onto $\R^n$, and endowed with a $\zeta_n$-structure. By claim \ref{moduli} the latter condition is equivalent to a lagrangian immersion $L^{m+k} \looparrowright M \times \C^n$. 
More precisely, the space $\Omega^{\infty-n} ({\mathcal{G}}(\tau \oplus \C^n)^{-\zeta_n})$ is uniquely defined by the property that given a smooth manifold $X$, the set of homotopy classes of maps $[X, \Omega^{\infty-n} ({\mathcal{G}}(\tau \oplus \C^n)^{-\zeta_n})]$, is in bijection with concordance classes over $X$, of smooth manifolds $E \subset X \times \R^\infty \times \R^n$, so that the first factor projection: $\pi : E \longrightarrow X$ is a submersion, and which are endowed with a smooth map $\varphi : E \longrightarrow M \times \C^n$ which restricts to a lagrangian immersion on each fiber of $\pi$. As before, we demand that the third factor projection $E \longrightarrow \R^n$ be fiberwise proper over $X$. 

\medskip
\noindent
In the above interpretation, given the standard inclusion $\R^{n_1} \subseteq \R^{n_2}$, the natural map: 
\[ \varphi_{n_1,n_2} : \Sigma^{n_1} {\mathcal{G}}(\tau \oplus \C^{n_1})^{-\zeta_{n_1}} \longrightarrow \Sigma^{n_2} {\mathcal{G}}(\tau \oplus \C^{n_2})^{-\zeta_{n_2}}, \] 
on the level of infinite loop spaces, represents the map that sends a concordance class $E$ to $E \times \R^{n_2-n_1}$, by simply taking the product with the orthogonal complement of $\R^{n_1}$ in $\R^{n_2}$. We thank David Ayala for patiently helping us understand this point of view. 

\begin{remark}
Recall from earlier discussion that a compact manifold $L$ admitting a lagrangian immersion $L \looparrowright M$ represents a point in $\Omega^{\infty} ({\mathcal{G}}(\tau)^{-\zeta_M})$. In particular, $L$ generates a point in the directed system. Notice however, that  given a (possibly non-compact) manifold $L$ immersing into $M$, one may still construct a point in this system, provided one has a $1$-form $\alpha$ on $L \times \R^n$ for some $\R^n$, with the property that $\alpha^\ast \iota_{\theta} : L \times \R^n \longrightarrow \R^n$ is proper. Here $\iota_{\theta}$ is the linear projection $T^\ast (L \times \R^n) \longrightarrow (\R^n)^\ast = \R^n$, which is being pulled back to $L \times \R^n$ along $\alpha$. In particular, any function $\phi$ on $L \times \R^n$ so that $\alpha = d \phi$ is proper over $V$ gives rise to a point in the directed system. Such functions $\phi$ are a natural analog of the theory of phase functions in our context \cite{W2}. 
\end{remark}

\begin{defn} \label{stablag}
Define the Thom spectrum $\Omega(M)$ representing the infinite loop space of stabilized lagrangian immersions in $M$ to be the colimit:
\[ \Omega(M) = \underline{\mathcal{G}}(M)^{-\zeta} := \mbox{colim}_n \, \, \Sigma^n {\mathcal{G}}(\tau \oplus \C^n)^{-\zeta_n} \]
Notice that by definition, we have a canonical homotopy equivalence: $\Omega(M \times \C) \simeq \Sigma^{-1} \Omega(M)$. 
\end{defn}

\begin{remark} \label{hpty}
The spectrum $\Omega(M)$ can also be described as a Thom spectrum: Let the stable tangent bundle of $M$ of virtual (complex) dimension $m$ be given by a map $\tau : M \longrightarrow \Z \times \BU$. As suggested by the notation, let $\underline{\mathcal{G}}(M)$ be defined as the pullback: 
\[
\xymatrix{
\underline{\mathcal{G}}(M)       \ar[d] \ar[r]^{\zeta \, \,  } & \Z \times \BO \ar[d] \\
M      \ar[r]^{\tau \quad} & \Z \times \BU.
}
\]
Then the spectrum $\Omega(M)$ is homotopy equivalent to the Thom spectrum of the stable vector bundle $-\zeta$ over $\underline{\mathcal{G}}(M)$ defined in the diagram above. 
\end{remark}

\smallskip
\noindent
Notice that the fibration $\Z \times \BO \longrightarrow \Z \times \BU$ is a prinicipal bundle up to homotopy, with fiber being the infinite loop space $\UO$. Hence, the spectrum $\Omega(M)$ is homotopy equivalent to a $(\UO)^{-\zeta}$-module spectrum. Here the bundle $\zeta$ over $\UO$ is the virtual zero dimensional bundle over $(\UO)$ defined by the canonical inclusion $\UO \longrightarrow \BO$. Now, observe that up to homotopy, we have the equivalence of $\UO$-spaces:
\[ \underline{\mathcal{G}}(M) \times_{\UO} \underline{\mathcal{G}}(N) \simeq \underline{\mathcal{G}}(M \times N). \]
The above equivalence translates to a canonical equivalence, up to homotopy:
\[ \mu : \Omega(M) \wedge_\Omega \Omega(N) \simeq  \Omega(M \times N), \]
where we introduce the notation $\Omega$ for the spectrum $(\UO)^{-\zeta}$.

\begin{example} \label{cotang}
Let us end this section with the example of the homotopy type of $\Omega(M)$ for a cotangent bundle $M = T^\ast X$, on a smooth $m$-dimensional manifold $X$ and endowed with the canonical symplectic form. Now since the tangent bundle $\tau$ of $M$ admits a lift to $\Z \times \BO$,  the space $\underline{\mathcal{G}}(T^\ast X)$ is homotopy equivalent to $(\UO) \times X$. In particular, stably, a compact lagrangian immersion $L \looparrowright T^\ast X$ is represented by a $\zeta \times \tau(X)$ structure on $L$ \footnote{Hence the stable nature of caustics (defined as the critical set of the projection map $\pi : L \rightarrow X$) is measured by a ``universal Maslov structure" on $L$, which we define as a bundle map from the stable fiberwise tangent bundle of $\pi$, to the bundle $\zeta$, that lifts the universal Maslov class $L \longrightarrow \UO$.}. Now it is easy to see that there is a homotopy equivalence:
\[ \Omega(T^\ast X) \simeq \Omega \wedge X^{-\tau(X)} \]
where $X^{-\tau(X)}$ denotes the Thom spectrum of the formal negative of $\tau(X)$. If $X$ were compact, then $\Omega \wedge X^{-\tau(X)}$ is equivalent to $\Map(X, \Omega)$. This can be interpreted as saying that a stable lagrangian $L$ in $T^\ast X$ is represented by a family of virtual dimension zero stable lagrangians parametrized as fibers of the map $\pi : L \rightarrow X$. Furthermore, the structure of $\Omega \wedge X^{-\tau(X)}$ as a ring spectrum provides us with an interesting way of multiplying the stable equivalence class of lagrangians in $T^\ast X$. This is the homotopical version of a $\ast$-algebra of Weinstein \cite{W2}.
\end{example}

\begin{remark}\label{totreal}
In this section if one drops the assumption of monotonicity, i.e. that $[\omega]$ is a scalar multiple of $c_1(M)$, then in the statement of claim \ref{str} lagrangian immersions must be replaced by totally real immersions. In this context, the space $\Omega^{\infty} (\Omega(M))$ geometrically represents the space of stabilized (as above) totally real immersions in $M$. 
\end{remark}

\section{The Stable Symplectic homotopy category} \label{hcat}

\noindent
In this section we describe a stabilization of Weinstein's symplectic category, with morphisms being enriched over the homotopy category of $\Omega$-module spectra. In later sections, we will describe an $A_\infty$-model for this category that is enriched over an honest category of $\Omega$-module spectra. 

\smallskip
\noindent
Let us now describe the stable symplectic homotopy category $h\mathbb{S}$. By definition, the objects of this category $h\mathbb{S}$, will be symplectic manifolds $(M,\omega)$ (see remark \ref{noncpct}), endowed with a compatible almost complex structure. The morphisms in our category $h\mathbb{S}$ will naturally have the structure of Thom spectra. Let $(M,\omega)$ and $(N,\eta)$ be two objects. We define the conjugate of $(M,\omega)$ to be the symplectic manifold $\overline{M}$ which has the same underlying manifold as $M$ but with symplectic form $-\omega$. 

\begin{defn}
The ``morphism spectrum" $\Omega(M,N)$ in $h\mathbb{S}$ between $M$ and $N$ is defined as the $\Omega$-module spectrum: 
\[ \Omega (M,N) := \Omega(\overline{M} \times N). \]
Observe that we have a canonical homotopy equivalence: $\Omega(M\times \C, N) = \Sigma^{-1} \Omega_{\zeta}(M,N)$. The same holds on replacing $N$ by $N \times \C$. Therefore, up to natural equivalence, the morphism spectra factor through the equivalence on symplectic manifolds defined in the introduction. 
\end{defn}

\begin{remark}\label{noncpct}
Notice that objects in $h\mathbb{S}$ are allowed to be non-compact symplectic manifolds. The price we pay for this, as we shall see later, is that we simply lose the identity morphisms for such objects.  Compact manifolds also enjoy other nice properties that fail for non-compact manifolds. We will therefore state explicitly when compactness is assumed. 
\end{remark}

\noindent
The next step is to define composition. The simplest case of composition is of the form:
\[ \Omega(M, \ast) \wedge_{\Omega} \Omega(\ast, N) \longrightarrow \Omega(M,N), \]
where $M$ and $N$ are two objects in $h\mathbb{S}$. This composition is defined to be the map $\mu$ constructed in the previous section, before the example \ref{cotang}. 

\smallskip
\noindent
For the general case, consider $k+1$ objects objects $M_i$ with $0 \leq i \leq k$, and let the space $\underline{\mathcal{G}}(\Delta)$  be defined by the pullback:
\[
\xymatrix{
\underline{\mathcal{G}}(\Delta)       \ar[d]^{\xi} \ar[r] &  \underline{\mathcal{G}}(\overline{M}_0 \times M_1 \times \cdots  \times \overline{M}_{k-1} \times M_k) \ar[d] \\
\overline{M}_0 \times (M_1 \times \cdots \times M_{k-1})  \times M_k    \ar[r]^{\Delta \quad \quad \quad \quad} &  \overline{M}_0 \times (M_1 \times \overline{M}_1) \times \cdots \times (M_{k-1} \times \overline{M}_{k-1}) \times M_k 
}
\]
where $\Delta$ denotes the product to the diagonals $\Delta : M_i \longrightarrow M_i \times \overline{M}_i$, for $0 < i < k$. 

\medskip
\noindent
Now notice that the fibrations defining the pullback above are direct limits of smooth fibrations with compact fiber. Furthermore, the map $\Delta$ is a proper map for any choice of $k+1$-objects (even if they are non-compact). In particular, we may construct the Pontrjagin--Thom collapse map along the top horizontal map by defining it as a direct limit of Pontrjagin--Thom collapses for each smooth stage. 

\smallskip
\noindent
Let $\zeta_i$ denote the individual structure maps $\underline{\mathcal{G}}(\overline{M} _{i-1} \times M_i) \longrightarrow \Z \times \BO$, and let $\eta(\Delta)$ denote the normal bundle of $\Delta$. Performing the Pontrjagin--Thom construction along the top horizontal map in the above diagram yields a morphism of spectra:
\[ \varphi : \Omega(M_0, M_1) \wedge_\Omega \cdots \wedge_\Omega \Omega(M_{k-1}, M_k) \simeq \Omega(\overline{M}_0 \times M_1 \times \cdots  \times \overline{M}_{k-1} \times M_k) \longrightarrow \underline{\mathcal{G}}(\Delta)^{-\lambda} \]
where $\lambda : \underline{\mathcal{G}}(\Delta)  \longrightarrow \Z \times \BO$ is the formal difference of the bundle $\bigoplus \zeta_i$ and the pullback bundle $\xi^* \eta(\Delta)$.

\medskip
\noindent
The next step in defining composition is to show that $ \underline{\mathcal{G}}(\Delta)^{-\lambda}$ is canonically homotopy equivalent to $\Omega(M_0, M_k) \wedge (M_1 \times \cdots \times M_{k-1})_+$, where $(M_1 \times \cdots \times M_{k-1})_+$ denotes the manifold $M_1 \times \cdots \times M_{k-1}$ with a disjoint basepoint. Recall that the map $\xi$ defined as the pullback above is a principal $\UO$-bundle, therefore, to achieve the equivalence we seek, it is sufficient to construct a $\UO$-equivariant map over $\overline{M}_0 \times (M_1 \times \cdots \times M_{k-1})  \times M_k$:
\[ \psi : \underline{\mathcal{G}}(\overline{M}_0 \times M_k) \times (M_1 \times \cdots \times M_{k-1}) \longrightarrow \underline{\mathcal{G}}(\Delta), \]
that pulls $\lambda$ back to the bundle $\zeta \times 0$. The construction of $\psi$ is straightforward.

We define:
\[ \psi( \lambda, m_1, \ldots, m_{k-1}) = \lambda \oplus \Delta(T_{m_1}(M_1)) \oplus \cdots \oplus \Delta(T_{m_{k-1}}(M_{k-1})), \]
where $\Delta(T_m(M)) \subset T_{(m,m)}(M \times \overline{M})$ denotes the diagonal lagrangian subspace. Now let $\pi : \underline{\mathcal{G}}(\Delta)^{-\lambda} \longrightarrow \Omega(M_0,M_k)$ be the projection map that collapses $M_1 \times \cdots \times M_{k-1}$ to a point. 

\begin{defn}
We define the composition map to be the induced composite:
\[ \pi \varphi :  \Omega(M_0, M_1) \wedge_\Omega \cdots \wedge_\Omega \Omega(M_{k-1}, M_k) \longrightarrow \underline{\mathcal{G}}(\Delta)^{-\lambda} \longrightarrow \Omega(M_0,M_k). \]
We leave it to the reader to check that composition as defined above is homotopy associative. In fact, this composition has more structure. The question of how structured this associative composition is will be addressed in the next sections. 
\end{defn}

\noindent
{\bf The identity morphism:}

\medskip
\noindent
We now assume that the object $(M, \omega)$ is a compact manifold. It is a natural question to ask if an identity morphism exists for such an object. 
\begin{claim}
Let $M$ be a compact manifold, and let $[id] : \mbox{S} \longrightarrow \Omega(M,M)$ denote the map that is a representative of the unit map upto homotopy representing the diagonal (lagrangian) embedding $\Delta : M \longrightarrow \overline{M} \times M$. Then $[id]$ is indeed the identity for the composition defined above, up to homotopy. Furthermore, $[id]$ factors through the unit map of the ring spectrum $M^{-\tau}$ (see \cite{C}). 
\end{claim}
\begin{proof}
The fact that the map representing the lagrangian embedding $\Delta$ factors through $M^{-\tau}$ is essentially the method of how one constructs a map $\mbox{S} \longrightarrow \Omega(M,M)$ given a lagrangian in $\overline{M} \times M$. Now given two manifolds $M,N$, let $\Delta(M) \subset \overline{M} \times M$ is the diagonal representative of $[id]$ as above. Observe that $\overline{N} \times \Delta(M) \times M$ is transverse to $\overline{N} \times M \times \Delta(M)$ inside $\overline{N} \times M \times \overline{M} \times M$. They intersect along $\overline{N} \times \Delta_3(M)$, where $\Delta_3(M) \subset M \times \overline{M} \times M$ is the triple (thin) diagonal. Hence we get a commutative diagram up to homotopy:
\[
\xymatrix{
\Omega(N, M) \wedge \mbox{S} \ar[dr] \ar[r] & \Omega(N, M) \wedge \Delta(M)^{-\tau} \ar[r]^{\Delta^{-\tau}} \ar[d] &  \Omega(N,M) \wedge_\Omega \Omega(M,M) \ar[d] \\
 & \Omega(N, M) \ar[r]^{=} & \Omega(N,M)}
\]
where the right vertical map is composition, and the left vertical map is the Pontrjagin--Thom collapse over the inclusion map $\overline{N} \times M = \overline{N} \times \Delta_3(M) \longrightarrow \overline{N} \times M \times \Delta(M)$. Now consider the following factorization of the identity map: 
\[ \overline{N} \times M = \overline{N} \times \Delta_3(M) \longrightarrow \overline{N} \times M \times \Delta(M) \longrightarrow \overline{N} \times M \]
where the last map is the projection onto the first two factors. Performing the Pontrjagin--Thom collapse over this composite shows that the following composite is the identity: 
\[ \Omega(N,M) \wedge \mbox{S} \longrightarrow \Omega(N,M) \wedge \Delta(M)^{-\tau} \longrightarrow \Omega(N,M). \]
This shows that right multiplication by $[id] : \mbox{S} \rightarrow \Omega(M,M)$ induces the identity map on $\Omega(N,M)$, up to homotopy. A similar argument works for left multiplication. 
\end{proof}

\medskip
\noindent
We may interpret the composition map of $h\mathbb{S}$ geometrically as follows: Given a collection of $k$ lagrangians $L_i$ for $1 \leq i \leq k$, let $Y = L_1 \times L_2 \times \cdots \times L_k$ be a product of stable lagrangians immersed in  $\overline{M}_0 \times (M_1 \times \overline{M}_1) \times \cdots \times (M_{k-1} \times \overline{M}_{k-1}) \times M_k$. Assume that $X \subset Y$ is the transverse intersection of $Y$ along $\overline{M}_0 \times \Delta(M_1) \times \cdots \times \Delta (M_{k-1}) \times M_k$. Then we have a homotopy commutative diagram:
\[
\xymatrix{
Y^{-\tau(Y)} \ar[r] \ar[d] & X^{-\tau(X)} \ar[d] \\
 \Omega_{\zeta_1}(M_0, M_1) \wedge_\Omega \cdots \wedge_\Omega \Omega_{\zeta_k}(M_{k-1}, M_k) \ar[r] & 
\Omega(M_0,M_k).}
\]
with the top row representing the Pontrjagin--Thom construction along the inclusion $X \subset Y$. It is easy to see that $X$ supports a lagrangian immersion into $\overline{M}_0 \times M_k$, represented by the composite: 
\[ [X] : \mbox{S} \longrightarrow X^{-\tau(X)} \longrightarrow \Omega(M_0,M_k). \]
Now let $\Omega^{\infty}(M,N)$ denote the infinite loop space $\Omega^\infty(\Omega(M,N))$. Then from the geometric standpoint on the symplectic category, it is more natural to consider the unstable composition map:
\[  \Omega^{\infty}(M_0, M_1) \times \cdots \times \Omega^{\infty}(M_{k-1}, M_k) \longrightarrow \Omega^{\infty}(M_0, M_k) \]
By applying $\pi_0$ to this morphism, and invoking the above observation, we see that the definition of composition is faithful to Weinstein's definition of composition in the symplectic category.

\begin{remark}
Recall that given arbitrary symplectic manifolds $M$ and $N$, there is a natural decomposition of $\Omega(M,N)$ in $h\mathbb{S}$ induced by the composition map:
\[ \Omega(M,\ast) \wedge_{\Omega} \Omega(\ast, N) \simeq \Omega(M,N). \]
In particular, arbitrary compositions can be canonically factored using the above decomposition, and computed by applying the following composition map internally to the factors: 
\[ \Omega(\ast, N) \wedge_{\Omega} \Omega(N, \ast) \longrightarrow \Omega. 
\]
\end{remark}

\section{Internal structure of the stable symplectic homotopy category} \label{four}

\noindent
Recall that $\Omega = (\UO)^{-\zeta}$ is a commutative ring spectrum. This spectrum has been studied in \cite{AG} (Section 2) and even earlier in \cite{SS}. In particular,
\[ \pi_{\ast} \Omega = \Z/2[x_{2i+1}, i \neq 2^k-1]. \]
In addition, the ring $\pi_\ast \Omega $ can be detected as a subring of $\pi_\ast \MO$. It follows that $\Omega $ is a generalized Eilenberg--Mac\,Lane spectrum over $\Hrt$. Since it acts on all morphism spectra $\Omega_{\zeta}(M,N)$, we see that all the spectra $\Omega (M,N)$ are also generalized Eilenberg--Mac\,Lane spectra (see the Appendix).

\begin{remark}
One has an oriented version of this category $hs\mathbb{S}$, obtained by replacing $\BO$ by $\BSO$. All definitions go through in this setting verbatim. However, the "coefficients" $s\Omega $ is no longer a generalized Eilenberg--Mac\,Lane spectrum. Its rational homology (or stable homotopy) is easily seen to be an exterior algebra on generators $y_{4i+1}$ in degrees $4i+1$, with $i \geq 0$:
\[ \pi_\ast s\Omega \otimes \Q = \Lambda (y_{4i+1} ). \] 

\end{remark}

\medskip

\begin{claim}\label{D}
Let $M$ be a compact symplectic manifold. Then the spectrum $\Omega(M, \ast)$ is canonically homotopy equivalent to $\Hom_{\Omega}(\Omega(\ast, M), \Omega)$. In other words, $\Omega_{\zeta}(M)$ is dual to $\Omega(\overline{M})$. 
\end{claim}
\begin{proof}
Consider the composition map: $\Omega(\ast, M) \wedge \Omega(M, \ast) \longrightarrow \Omega$. Taking the adjoint of this map yields the map we seek to show is an equivalence:
\[ \Omega(M, \ast) \longrightarrow \Hom_{\Omega}(\Omega(\ast, M), \Omega). \]
To construct a homotopy inverse to the above map, one uses the identity morphism:
\[  \Hom_{\Omega}(\Omega(\ast, M), \Omega) \wedge \mbox{S} \longrightarrow  \Hom_{\Omega}(\Omega(\ast, M), \Omega) \wedge_\Omega \Omega_{\zeta}(M,M) \longrightarrow \Omega(M, \ast), \]
where the last map is evaluation, once we identify $\Omega(M,M)$ with $\Omega(M,\ast) \wedge_{\Omega} \Omega(\ast, M)$. Details are left to the reader. 
\end{proof}
\begin{remark}
This is really Poincar\'e duality in disguise that says the the dual of the Thom spectrum of a bundle $\xi$ over $M$ is the Thom spectrum of $-\xi - \tau$, where $\tau$ is the tangent bundle of $M$. So the bundle that is ``self dual" is the bundle $\xi$ so that $\xi = -\xi -\tau$, i.e. $2\xi = -\tau$. The bundle $-\zeta$ is the universal bundle that satisfies this condition. 
\end{remark}

\noindent
The next theorem follows easily the definitions, and the above results:
\begin{thm} \label{Weyl1}
For arbitrary symplectic manifolds $M$ and $N$, there are canonical equivalences of $\Omega$-module spectra:
\[ \Omega(\ast, \overline{M} \times N) \simeq \Omega(M \times \overline{N}, \ast) \simeq \Omega(\overline{N},\overline{M}) \simeq \Omega(M,N). \]
Furthermore if $M$ is compact, then using duality on $M$, we have a canonical homotopy equivalence of $\Omega$-module spectra:
\[  \Omega(M, N) \simeq \Hom_{\Omega}(\Omega(\ast, M), \Omega(\ast, N)) \]
which is compatible with composition in $h\mathbb{S}$. In particular, for a compact manifold $M$, the ring spectrum $\Omega(M,M)$ has the structure of an endomorphism algebra up to homotopy \footnote{See section \ref{stab} for an interpretation of this algebra.}:
\[ \Omega(M,M) \simeq \End_{\Omega} \left( \Omega(\ast, M) \right). 
\]
All this structure holds for $s\Omega(M,M)$ as well. Furthermore, in the unoriented case, the above theorem may be strengthened (see theorem \ref{Weyl}). 
\end{thm}

\medskip
\noindent
{\bf Thom classes and an algebraic representation:} 

\medskip
\noindent
The next item on the agenda is to construct Thom classes. Notice from remark \ref{hpty} that the Thom spectra $\Omega(M,N)$ admit maps to $\Sigma^{-(m+n)} \MO$ that classify the map induced by the virtual bundle $-\zeta$, and are canonical up to homotopy. Similarly, morphisms in the oriented categoriy $hs\mathbb{S}$ admit canonical maps to $\MSO$. Working with $h\mathbb{S}$ for simplicity, let $\Fr$ now be a spectrum with the structure of a commutative algebra over the spectrum $\MO$. This allows us to obtain $\Omega$-equivariant Thom classes:
\[ \Fr(M,N) : \Omega(M,N) \longrightarrow \Sigma^{-(m+n)} \MO \longrightarrow \Sigma^{-(m+n)} \Fr. \] 

\begin{remark}
Given a choice of Thom classes, any object $N$ in $h\mathbb{S}$ has an induced $\Fr$ orientation: $N^{-\tau} \longrightarrow \Omega(N,N) \longrightarrow \Sigma^{-2n} \Fr$, induced by a lift of the diagonal inclusion $N \longrightarrow \overline{N} \times N$. 
\end{remark}

\noindent
The Thom isomorphism theorem is now a formal consequence of the definitions. Consider the diagonal map: $\Omega(M,N) \longrightarrow \Omega(M,N) \wedge (\overline{M} \times N)_+$,
where $(\overline{M} \times N)_+$ denotes the space $\overline{M} \times N$ with a disjoint base point. This allows us to construct Thom maps in $\Fr$-homology and cohomology, given by capping with the Thom class, and cupping with it respectively. The Thom isomorphism now follows by an easy argument:
\begin{claim} \label{thom}
Let $M$ and $N$ be any two symplectic manifolds of dimension $2m$ and $2n$ resp. Given $\Omega$-equivariant Thom classes as above, there are canonical Thom isomorphisms:
\[ \pi_\ast (\Omega(M,N) \wedge_{\Omega} \Fr) = \Fr_{\ast+m+n}(\overline{M} \times N), \quad \pi_\ast  \Hom_{\Omega}(\Omega(M, N), \Fr) = \Fr^{-\ast+m+n} (\overline{M} \times N). \]
\end{claim}

\noindent
The following is now an easy consequence of \ref{Weyl1}:
 
 \begin{thm} \label{rep}
There exists an (algebraic) representation of the stable symplectic homotopy category in the category of $\pi_\ast \E$-modules:
\[ q : \E_{\ast+m} (M) \otimes \pi_\ast \Omega(M,N) \longrightarrow \E_{\ast+n} (N). \]
Furthermore, one has an intersection pairing : $\E_{\ast+m}(M) \otimes \E_{\ast+m} (\overline{M}) \longrightarrow \E_\ast$, which is non-degenerate for compact manifolds $M$. 
All of this structure also holds for the oriented case. In \cite{KM} we will show that, working over $\Q$, a group isomorphic to an abelian quotient of the Grothendieck-Teichm\"uller group, acts via monoidal automorphisms on this representation (see question \ref{GT}). 
\end{thm}

\section{The algebra of observables and the Sympletomorphism Group}\label{stab}
\noindent
We work with the category $h\mathbb{S}$ in this section. All constructions apply equally well for the oriented setting. Let $(M,\omega)$ be an object. It is our interest to motivate the claim that $\Omega(M,M)$ is a homotopical version of the ``algebra of observables" of $M$. In deformation quantization, the algebra of observables is a non-commutative deformation of the algebra of functions $\mbox{C}^\infty(M)$, compatible with the Poisson structure on $M$. Now it is well known that functions in $\mbox{C}^\infty(M)$ can be identified with a subspace of lagrangians in $T^\ast M$ by mapping a function $f$ to the graph of $df$. Hence, in our context, the (commutative) ring spectrum $\Omega(T^\ast M) \simeq \Omega \wedge M^{-\tau}$ is the analog of the ring $\mbox{C}^\infty(M)$ (see example \ref{cotang}). The analog of the algebra of observables should therefore be an (associative) ring spectrum which supports a map of associative spectra from $\Omega \wedge M^{-\tau}$. We shall demonstrate in section \ref{comp} (see remarks \ref{dq} and \ref{dq2}), that the spectrum $\Omega(M,M)$ is precisely such a spectrum. 

\medskip
\noindent
The purpose of this section is to extend the correspondence between $\mbox{C}^\infty(M)$ and $T^\ast M$, to a correspondence between the symplectomorphism group of $M$, $\Symp(M)$ and $\Omega(M,M)$. As before, this correspondence simply takes takes a symplectomorphism to its graph seen as a lagrangian in $\Omega(M,M)$. This will make sense if $M$ is a compact manifold. In fact, for a compact symplectic manifold $M$, we will show that the identification of a symplectomorphism with its graph can be de-looped to a canonical map: 
\[ \gamma : \BSymp(M) \longrightarrow \BGL  \Omega(M,M), \]
where $ \GL  \Omega(M,M)$ is defined as the space of components in the ring-space: $\Omega^\infty(\Omega(M,M))$, that represent units in the ring $\pi_0(\Omega(M,M))$.

\medskip
\noindent
Let us begin for the moment with an arbitrary symplectic manifold $(M^{2m},\omega)$. We will construct a map 
$ \gamma : \BSymp(M) \longrightarrow \BAut _\Omega(\Omega(M))$. If $M$ is compact, then by the results of section \ref{comp}, one has an $A_\infty$-equivalence $\Omega(M,M) \longrightarrow \End_\Omega(\Omega(M))$. In particular, may identify $\BGL \Omega(M,M)$ with $\BAut_\Omega(\Omega(M))$. 

\smallskip
\noindent
The first step in the construction of the map $\gamma$ is to observe that we may construct the spectrum $\underline{\mathcal{G}}(M)^{-\zeta}$ fiberwise over $\BSymp(M,\omega)$. More precisely, consider the space $\mathcal{J}(M)$ consisting of pairs $(J,m)$ with $m \in M$ and $J$ a compatible complex structure on $(M,\omega)$. This space fibers over the space of compatible complex structures on $(M,\omega)$, with fiber $M$. In particular, it is homotopy equivalent to $M$. The space $\mathcal{J}(M)$ supports a canonical unitary vector bundle $\mathcal{J}(\tau)$, whose fiber over a point $(m,J)$ is the tangent space $T_m(M)$ endowed with the complex structure $J$. Notice that the symplectomorphism group acts on the space $\mathcal{J}(\tau)$ by unitary bundle automorphisms. It follows that $\mathcal{J}(\tau)$ extends to a unitary vector bundle $\Jb(\tau)$ over the associated bundle:
\[  \Jb(M) := \ESymp(M,\omega) \times_{\Symp} \mathcal{J}(M). \]
Consider the pullback $\underline{\mathcal{G}}(\Jb(M))$ fibering over $\BSymp(M)$ defined as a pullback:
\[
\xymatrix{
\underline{\mathcal{G}}(\Jb(M))     \ar[d]^{\xi} \ar[r]^{\Jb(\zeta)} & \Z \times \BO \ar[d] \\
 \Jb(M)    \ar[r]^{\Jb(\tau)}& \Z \times \BU
}
\]
Notice that the restriction of the bundle $\Jb(\tau)$ to any any subspace $(M,J)$ in the fiber over $\BSymp(M)$ yields the tangent bundle $TM$ with complex structure $J$. Similarly, the restriction of $\Jb(\zeta)$ to $\underline{\mathcal{G}}(M)$ is the bundle $\zeta$. 

\medskip
\noindent
One may construct the fiberwise Thom spectrum: $\underline{\mathcal{G}}(\Jb(M))^{-\Jb(\zeta)}$ is a bundle of $\Omega$-module spectra over $\BSymp(M)$, with fiber being homotopy equivalent to $\Omega(M)$. We may classify this bundle by a map: $\gamma : \BSymp(M) \longrightarrow \BAut_\Omega(\Omega(M))$. 

\medskip
\noindent
Now if $M$ is a compact manifold, then recall that $\Omega(M,M)$ is equivalent to $\End_\Omega(\Omega(M))$ as $A_\infty$-ring spectra. Hence for a compact symplectic manifold $(M,\omega)$, one gets a map: 
\[\gamma : \BSymp(M) \longrightarrow  \BGL  \Omega(M,M). \]

\medskip
\noindent
It remains to identify the map $\Omega(\gamma) : \Symp(M) \longrightarrow \GL  \Omega(M,M)$ as the map that takes an element $\varphi \in \Symp(M)$ to its graph in $\Omega(M,M)$. Let us push forward to $\Aut_\Omega(\Omega(M))$. The composite map:
\[ \Omega(\gamma) : \Symp(M) \longrightarrow \Aut_\Omega(\Omega(M)), \]
is easily seen to be the map that sends an element $\varphi \in \Symp(M)$ to the left action of $\varphi$ on $\Omega(M)$ which has the effect of sending a lagrangian immersion $L \looparrowright M$ to the immersion $\varphi(L) \looparrowright M$. We need to identify this left action with the graph of $\varphi$ as an element in $\Omega(M,M)$. Now the graph of a symplectomorphism $gr(\varphi) \in \overline{M} \times M$, and an arbitrary lagrangian $L$ in $M$, it is easy to check that the product $L \times gr(\varphi)$ is always transversal to the submanifold $ \Delta(M) \times M \subset  M \times \overline{M} \times M$. It follows easily from this that the graph of $\varphi$ acts exactly as the left action of $\varphi$ under the composition $\Omega(\ast,M) \wedge_\Omega \Omega(M,M) \longrightarrow \Omega(\ast, M)$. Details are left to the reader. 

\begin{remark}
Recall that we have a canonical equivalence: $\Omega(M \times \C) \simeq \Sigma^{-1} \Omega(M)$. It follows from this observation that the map $\gamma : \BSymp(M) \longrightarrow \BAut_\Omega(\Omega(M))$ factors through $\BSymp_s(M)$, where we define $\Symp_s(M)$ as the stabilization of the symplectomorphism group:
\[ \Symp_s(M) = \colim_k \, \Symp(M \times \C^k). \]

\end{remark}

\newpage

\section{The Stable Symplectic category as an $A_\infty$-category: Background} \label{A-mor}

\noindent
From now on we will be interested in turning the stable symplectic homotopy category into an $A_\infty$-category. This category will be enriched over a suitable category of structured spectra so that one can actually perform composition on the point-set level as it were, and not just up to homotopy. The framework that works best for our needs is that of parametrized $\mbox{S}$-modules \cite{EKMM}, \cite{MS}. This framework gives us an honest monoidal smash product on the category of spectra, and allows us to define ring spectra that are rigid enough to allow one to perform standard algebraic constructions. We shall use the language of \cite{EKMM} freely from this point on, with apologies to the reader unfamiliar with it. The virtue of this language is that it allows us to describe the stable symplectic category as an $A_{\infty}$-category. The reader unfamiliar with this language may safely work at the homotopical level and ignore the decoration by the operad $\mathcal{L}$ throughout. Let us briefly recall the basic setup:

\smallskip
\noindent
The spectra we study in this article are naturally indexed on the universe of euclidean subspaces of $\C^\infty$. As such there are two natural choices of isometries that act on these subspaces:

\begin{defn}
The space of linear isometries $\mathcal{L}^\R(k)$ is defined as $\mathcal{L}^\R(k) = \mathcal{I}((\R^{\infty})^k, \R^{\infty})$, 
where $\mathcal{I}((\R^{\infty})^k, \R^{\infty})$ denotes the contractible space of isometries from $k$-copies of $\R^\infty$, to $\R^\infty$. Similarly, define $\mathcal{L}^\C(k) = \mathcal{I}((\C^{\infty})^k, \C^{\infty})$ as the space of unitary isometries from the Hermitian space $(\C^\infty)^k$, to $\C^\infty$. Recall \cite{EKMM} that the spaces $\mathcal{L}^\R(k)$, and $\mathcal{L}^\C(k)$ naturally form an $E_\infty$-operad. Notice also that there is a natural inclusion of operads given by complexification: $\mathcal{L}^\R(k) \subset \mathcal{L}^\C(k)$. 
\end{defn}

\smallskip
\noindent
Composition of isometries turns the spaces $\mathcal{L}^\R(1)$ and $\mathcal{L}^\C(1)$ into monoids. Since the spectra we study are naturally indexed on the universe of euclidean subspaces of $\C^\infty$, these monoids act on this collection of subspaces. The following claim is straightforward:

\begin{claim}
Let $\BU(k)$ denote the classifying space of $\U(k)$ defined as the Grassmannian of complex $k$-planes in $\C^\infty$. Consider the map $\pi : \mathcal{L}^\C(1) \longrightarrow \BU(k)$, that sends an isometry to its image on the standard subspace $\C^k \subset \C^\infty$. Then $\pi$ is a principal bundle, with fiber given by the sub-monoid $\mathcal{L}_m^\C(1) \times \U(k) \subseteq \mathcal{L}^\C(1)$, where $\mathcal{L}_m^\C(1)$ is the space of isometries that fix the standard subspace $\C^k$ pointwise. Similarly, one has a principal $\mathcal{L}_m^\R(1) \times \Or(k)$ bundle $\mathcal{L}^\R(1) \longrightarrow \BO(k)$. 
\end{claim}

\smallskip
\noindent
Now let $(M^{2m},\omega)$ denote a symplectic manifold endowed with a compatible complex structure. Assume we are given a map $\tau : M \longrightarrow \BU(m)$ classifying the tangent bundle.
\begin{defn}
Define the extended frame bundle $\pi : \mathcal{F}(M) \longrightarrow M$ as the pullback:
\[
\xymatrix{
\mathcal{F}(M)    \ar[r] \ar[d]^\pi & \mathcal{L}^\C(1) \ar[d]^{\pi} \\
M \ar[r]^{\tau \quad} & \BU(m). 
}
\]
In particular, $\mathcal{F}(M)$ is a principal $\mathcal{L}_m^\C(1) \times \U(m)$-bundle that admits a compatible map to $\mathcal{L}^\C(1)$. 

\end{defn}

\noindent
Let $\mathbb{L}^\R$, and $\mathbb{L}^\C$ be the monads generated by sending a spectrum $X$ indexed over the universe $\C^\infty$ to the ``free spectrum" $\mathbb{L}^\R(X) := \mathcal{L}^\R(1) \ltimes X$, and $\mathbb{L}^\C(X) = \mathcal{L}^\C(1) \ltimes X$ respectively, as described in \cite{EKMM}. Let $\mathbb{L}$ stand for either $\mathbb{L}^\C$ or $\mathbb{L}^\R$. We shall say that a spectrum $X$ is an $\mathbb{L}$-spectrum (or an $\mathcal{L}(1)$-spectrum) if it is an algebra over the monad $\mathbb{L}$. By definition, an $\mathbb{L}$-spectrum admits a map of spectra $\xi : \mathbb{L}(X) \longrightarrow X$ that makes the following diagram commute:
\[
\xymatrix{
\mathbb{L} \mathbb{L}(X)     \ar[r]^{\mathbb{L}(\xi)} \ar[d]^{\mu} & \mathbb{L}(X) \ar[d]^{\xi} \\
\mathbb{L}(X)  \ar[r]^{\xi} & X, 
}
\]
where $\mu$ is the monad structure on $\mathbb{L}$. A standard example of an $\mathbb{L}$-spectrum (or algebra over the monad $\mathbb{L}$) is given by the sphere spectrum indexed on $\C^\infty$, which we henceforth denote by $\mbox{S}$. Given two $\mathbb{L}$-spectra $X$ and $Y$, we shall borrow the notation from \cite{EKMM}: 
\[ X \wedge_\mathcal{L} Y := \mathcal{L}(2) \ltimes_{\mathcal{L}(1) \times \mathcal{L}(1)} (X \wedge Y). \] 
A final piece of notation is that of an $\mbox{S}$-module. Given an $\mathbb{L}$-spectrum $X$, we say $X$ is an $\mbox{S}$-module, if the following canonical map is an isomorphism of $\mathbb{L}$-spectra: $\lambda : \mbox{S} \wedge_{\mathcal{L}} X \longrightarrow X$. An important example of an $\mbox{S}$-module is the sphere spectrum $\mbox{S}$ itself. In particular, given any $\mathbb{L}$-spectrum $X$, the spectrum $\mbox{S} \wedge_\mathcal{L} X$ is an $\mbox{S}$-module.

\begin{conv}
Since the decoration by $\R$ and $\C$ can introduce unnecessary clutter, let us set some notation going forward. Unless otherwise stated, the notation $\mathbb{L}$ will denote $\mathbb{L}^\C$. In constrast, $\mathcal{L}(k)$ will denote $\mathcal{L}^\R(k)$. In the the other cases, we continue to use the notation $\mathbb{L}^\R$ and $\mathcal{L}^\C(k)$. 
\end{conv}

\medskip
\noindent
{\bf The $\mathcal{L}(1)$-spectrum $(\UO)^{-\zeta}$}

\medskip
\noindent
An important example of an $\mathcal{L}(1)$-spectrum is the spectrum $(\UO)^{-\zeta}$. Let us describe the explicit model of $(\UO)^{-\zeta}$ that we will use. Given a $k$ dimensional subspace $V \subset \R^\infty$, let $V_\C$ denote its complexification: $(V \oplus iV) \subset \C^\infty$. The Grassmannian of lagrangian planes in $V_\C$ can be identified with the homogeneous space $\U(V_\C)/\Or(V)$, seen as the orbit of the group $\U(V_\C)$ of unitary transformations on $V_\C$ acting on the standard lagrangian $V \subset V_\C$. The stabilizer of $V$ under this action is the group $\Or(V)$ of orthogonal transformations of $V$. This Grassmannian supports a universal vector bundle $\zeta_V$ whose fiber over a lagrangian subspace $L \in \U(V_\C)/\Or(V)$, is the space of vectors in $L$. This vector bundle $\zeta_V$ includes into the trivial bundle $V_\C$, and let $\eta_V$ denote its normal bundle. 

\medskip
\noindent
We define $\U(V_\C)/\Or(V)^{-\zeta_V}$ as the desuspended Thom spectrum of the bundle $\eta_V$ defined as $\Sigma^{-V_\C} \U(V_\C)/\Or(V)^{\eta_V} := \mbox{S}^{-V_\C} \wedge \U(V_\C)/\Or(V)^{\eta_V}$, where $ \mbox{S}^{-V_\C}$ is the spectrum representing the functor that evaluates an arbitrary spectrum on the vector space $V_\C$ \cite{EKMM}. 
Now notice that if $W \subseteq V$ is an inclusion of subspaces in $\R^\infty$, with $Z$ being the complement of $W$ in $V$, we may take the sum of a lagrangian space in $W_\C$, with $Z$ to get a lagrangian subspace in $V_\C$. This yields a map $\U(W_\C)/\Or(W) \longrightarrow \U(V_\C)/\Or(V)$. Furthermore, the restriction of $\zeta_V$ along this map is the bundle $\zeta_W \oplus Z$. It follows that we have a canonical map $\Sigma^W \U(W_\C)/\Or(W)^{-\zeta_W} \longrightarrow \Sigma^V \U(V_\C)/\Or(V)^{-\zeta_V}$ which is compatible with respect to inclusions. 

\begin{remark}
It is important to observe that an element $\varphi \in \mathcal{L}(1)$ naturally identifies the homogeneous space $\U(V_\C)/\Or(V)$ with $\U(\varphi(V)_\C)/\Or(\varphi(V))$. This extends to an action of $\mathcal{L}(1)$ on the collection of homogeneous spaces, and therefore on the corresponding Thom spectra described above. Notice in contrast, that there is no such action of the monoid $\mathcal{L}^\C(1)$ on the collection of homogeneous spaces, since $\mathcal{L}^\C(1)$ does not preserve subspaces of the form $V_\C$. 
\end{remark}

\begin{defn}
Define $(\UO)^{-\zeta}$ as the directed colimit of Thom-spectra:
\[ (\UO)^{-\zeta} = \colim_{V} \Sigma^V  \U(V_\C)/\Or(V)^{-\zeta_V}, \]
where the colimit runs over the poset of all finite dimensional subspaces $V$ of $\R^\infty$. It is clear from the construction that $(\UO)^{-\zeta}$ is an $\mathcal{L}(1)$-spectrum. In the this section, $\Omega$ is defined to be the $\mbox{S}$-module $\Omega := \mbox{S} \wedge_{\mathcal{L}} (\UO)^{-\zeta}$. It is also easy to see that $\Omega$ is in fact a commutative $\mbox{S}$-algebra \cite{EKMM}. 
\end{defn}

\smallskip
\noindent
Now, given a symplectic manifold $(M^{2m},\omega)$ endowed with a compatible complex structure and a classifying map $\tau : M \longrightarrow \BU(m)$, let us define:
\begin{defn}
By virtue of the map $\mathcal{F}(M) \longrightarrow \mathcal{L}^\C(1)$, we may define a fiberwise $\mathbb{L}$-spectrum $\mathcal{S}(M)$ as the spectrum parametrized over $\mathcal{E}(M)$: 
\[ \mathcal{S}(M) = \mathcal{L}^\C(1) \ltimes_{\mathcal{L}_m^\C(1) \times \U(m)} (\mathcal{F}(M) \ltimes \mbox{S}^{-m}),\]
with $\mbox{S}^{-m}$ denoting the desuspended sphere spectrum $\mbox{S}^{-\R^m}$. 
\end{defn}

\noindent
In the above discussion we required a choice of map $\tau : M \longrightarrow \BU(m)$ representing the tangent bundle of $M$. The universal construction should therefore be made over the space of all unitary injections (see remark \ref{unit}). We define $\mathcal{E}(M)$ to be the (contractible) space of all unitary bundle maps $TM \longrightarrow M \times \C^\infty$ that factor through some $\C^k$. One therefore has a canonical bundle $\mathcal{F}(M)$ over $\mathcal{E}(M)$. 

\begin{defn}
Define the $\Omega$-module spectrum parametrized over $\mathcal{E}(M)$ by: 
\[ \mathbb{G}(M) = \mathcal{S}(M) \wedge_{\mathcal{L}} \Omega := \mathcal{L}(2) \ltimes_{\mathcal{L}(1) \times \mathcal{L}(1)} (\mathcal{S}(M) \wedge \Omega). \]
Let $\pi^M$ denote the map that collapses $\mathcal{E}(M)$ to a point, inducing a functor $\pi^M_!$ from the category of spectra parametrized over $\mathcal{E}(M)$, to spectra indexed over $\C^\infty$ \cite{MS}. In consistency with previous notation, define the spectrum $\Omega(M)$: 
\[ \Omega(M) := \pi^M_! (\mathbb{G}(M)). \]
Furthermore, by the natural adjunctions of the functor $\pi^M_!$ (\cite{MS}, Ch.11), it is clear that $\Omega(M)$ is a (usual) $\Omega$-module spectrum. 
\end{defn} 

\noindent
The next claim demonstrates that the spectrum we have constructed above is faithful to the geometric object we studied in previous sections. 
\begin{claim}
The space $\Omega^{\infty}(\Omega (M))$ has the weak homotopy type of the space of stable totally real immersions into $M$. If the monotonicity condition holds for $M$, then this is indeed the space of stable lagrangian immersions.  
\end{claim}
\begin{proof}
Let us fix a choice of the map $\tau : \mathcal{F}(M) \longrightarrow \mathcal{L}^\C(1)$. Let us also fix an invertible isometry $\gamma : \R^\infty \times \R^\infty \longrightarrow \R^\infty$. The map $\tau$ above, along with $\gamma$ induce a weak homotopy equivalences over $M$:
\[ \{\gamma \} \ltimes ((\{ \tau \} \ltimes \mbox{S}^{-m}) \wedge \Omega) \longrightarrow \mathbb{G}(M). \]
Now by the definition of the pullback diagram definint the map $\tau$, we see that the fiber of the bundle $\mathcal{F}(M)$, over a point $x \in M$ can be identified with isometries that map $\C^k$ to $T_x(M)$. Hence the bundle $(\{ \tau \} \ltimes \mbox{S}^{-m}) \wedge (\UO)^{-\zeta}$, maps canonically via a homotopy equivalence, to the the bundle of spectra whose fiber at $x$ is the Thom spectrum of the (negative) canonical bundle on the space of Lagrangians in $T_x(M) \oplus \C^\infty$. Taking the push forward along $\pi^M_!$, we recover (up to homotopy) the spectrum $\underline{\mathcal{G}}(M)^{-\zeta}$ studied in the previous section, thereby proving the claim. 
\end{proof}

\medskip
\noindent
{\bf The internal product:}

\medskip
\noindent
Now let $(M^{2m},\omega)$ and $(N^{2n},\eta)$ be two symplectic manifolds endowed with the relevant structure. Our next objective is to describe a natural map from the product of the spectra $\Omega(M)$ and $\Omega(N)$ to the spectrum: $\Omega(M \times N)$. 

\noindent 
By definition it is easy to see that one has an equality:
\[ \mathbb{G}(M) \wedge_{\mathcal{L}} \mathbb{G}(N) = \mathcal{L}(4) \ltimes_{\mathcal{L}(1) \times \mathcal{L}(1) \times \mathcal{L}(1) \times \mathcal{L}(1)} (\mathcal{S}(M) \wedge \mathcal{S}(N) \wedge \Omega \wedge \Omega). \]
We may write the right hand side as: 
\[ (\mathcal{L}(2) \ltimes_{\mathcal{L}(1) \times \mathcal{L}(1)} (\mathcal{S}(M) \wedge \mathcal{S}(N)) \wedge_{\mathcal{L}} (\Omega \wedge_{\mathcal{L}} \Omega). \]

\medskip
\noindent
Assume we are given an orthogonal {\em isomorphism} $\gamma : \R^\infty \times \R^\infty \longrightarrow \R^\infty$, 
which identifies the subspace: $\R^m \times \R^n$ with $\R^{(m+n)}$ diagonally. Notice that the complexification of $\gamma$ induces an obvious map $\mathcal{E}(\gamma) : \mathcal{E}(M) \times \mathcal{E}(N) \longrightarrow \mathcal{E}(M \times N)$. Furthermore, the map $\gamma$ induces a map of bundles:
\[
\xymatrix{
\mathcal{F}(M) \times \mathcal{F}(N)      \ar[d] \ar[r]^{\gamma_\ast} &  \mathcal{F}(M \times N) \ar[d] \\
\mathcal{E}(M) \times \mathcal{E}(N)     \ar[r]^{\quad \mathcal{E}(\gamma)} &  \mathcal{E}(M \times N),
}
\]
compatible with the ``diagonal" maps of monoids:
\[ \gamma_\ast : \U(m) \times \U(n) \longrightarrow \U(m+n), \quad \quad \gamma_\ast : \mathcal{L}_m^\C(1) \times \mathcal{L}_n^\C(1) \longrightarrow \mathcal{L}_{(m+n)}^\C(1). \]
The map $\gamma_\ast$ is easily seen to induce a map of spectra parametrized over the map $\mathcal{E}(\gamma)$: 
\[ \gamma_\ast : \mathcal{L}(2) \ltimes_{\mathcal{L}(1) \times \mathcal{L}(1)} (\mathcal{S}(M) \wedge \mathcal{S}(N)) \longrightarrow \mathcal{L}(1) \ltimes_{\mathcal{L}(1)} \mathcal{S}(M \times N), \]
which extends to a map of $\Omega$-spectra parametrized over $\mathcal{E}(\gamma)$:
\[ \mathbb{G}(\gamma)_\ast : \mathbb{G}(M) \wedge_\Omega \mathbb{G}(N) \longrightarrow \mathbb{G}(M \times N). \]

\begin{claim} \label{int}
Given two symplectic manifolds, and a map $\gamma$ as above, the map $\gamma_\ast$ induces a homotopy equivalence of $\Omega$-module spectra called the internal product: 
\[ \gamma_\ast : \Omega(M) \wedge_{\Omega} \Omega(N) \longrightarrow \Omega(M \times N) := \pi^M_!(\mathbb{G}(M \times N)). \] 
\end{claim}
\begin{proof}
By construction $\gamma_\ast : \mathbb{G}(M) \wedge_\Omega \mathbb{G}(N) \longrightarrow \mathbb{G}(M \times N)$ is a fiberwise homotopy equivalence. It remains to identify $\Omega(M) \wedge_{\Omega} \Omega(N)$ with $\pi^{M \times N}_!(\mathbb{G}(M) \wedge_{\Omega} \mathbb{G}(N))$. Now from the the definition of $\pi^{M \times N}_!$, it is clear that $\pi^{M \times N}_!(\mathbb{G}(M) \wedge_\mathcal{L} \mathbb{G}(N)) = \Omega(M) \wedge_\mathcal{L} \Omega(N)$. Furthermore, since $\pi^{M \times N}_!$ is a pushout, it preserves preserves colimits. In addition, the adjunction properties of $\pi^{M \times N}_!$ as described in \cite{MS} (Ch.11) show that $\pi^{M \times N}_!$ turns the following fiberwise coequalizer diagram:
\[ \mathbb{G}(M)\wedge_\mathcal{L} \Omega \wedge_\mathcal{L} \mathbb{G}(N) \Longrightarrow \mathbb{G}(M) \wedge_\mathcal{L} \mathbb{G}(N) \longrightarrow \mathbb{G}(M) \wedge_{\Omega} \mathbb{G}(N), \]
into exactly the one that defines the smash product of $\Omega(M)$ and $\Omega(N)$ over $\Omega$. 
\end{proof}

\section{The Stable Symplectic category as an $A_\infty$-category: Morphisms} \label{comp}


\noindent
Let us now describe the $A_\infty$-version of the Stable Symplectic category $\mathbb{S}$. By definition, the objects of this category $\mathbb{S}$, will be symplectic manifolds $(M,\omega)$ (see remark \ref{noncpct}), endowed with a compatible almost complex structure. Let $(M,\omega)$ and $(N,\eta)$ be two objects. We define the conjugate of $(M,\omega)$ to be the symplectic manifold $\overline{M}$ which has the same underlying manifold as $M$ but with symplectic form $-\omega$. 
\begin{defn}
The ``morphism spectrum" $\Omega(M,N)$ in $\mathbb{S}$ between $M$ and $N$ is defined as the $\Omega$-module spectrum (see also remark \ref{alt}): 
\[ \Omega (M,N) := \Omega(\overline{M}) \wedge_\Omega \Omega(N). \]
Notice that by claim \ref{int}, the infinite loop space underlying $\Omega(M,N)$ represents the space of stable totally real immersions into $\overline{M} \times N$. 
\end{defn}

\noindent
The next step is to define composition in $\mathbb{S}$. As is to be expected in this framework, $\mathbb{S}$ will be an $A_{\infty}$- category enriched over the category of modules over the commutative $\mbox{S}$-algebra $\Omega$. The natural model for this operad will be the linear isometries. 

\medskip
\noindent
Given $k+1$-pairs of objects $(M_0,M_1), \ldots, (M_k,M_{k+1})$, our construction of composition in the category $\mathbb{S}$ will amount to a map of the form:
\[ \mathcal{O}_k : \Omega(M_0, M_1) \wedge_\Omega \cdots  \wedge_\Omega \Omega(M_k, M_{k+1}) \longrightarrow \Omega(M_0, M_{k+1}). \]
Indeed, our construction will also show such maps are naturally parametrized by a contractible space $\mathcal{O}(M)$ of choices on each object $M$. We make this precise later (see \ref{alt}).

\medskip
\noindent
{\bf The special case of composition:}

\medskip
\noindent
The next step towards defining composition in general is to define the important case :
\[ \mathcal{O} : \mathcal{O}(M)_+ \wedge \Omega(\ast, M) \wedge_\Omega \Omega(M, \ast) \longrightarrow \Omega(\ast, \ast), \]
 The space $\mathcal{O}(M)$, and the map $\mathcal{O}$ will be defined below. This special case describes the basic idea behind composition. The strategy then is to use the special case repeatedly in extending composition to the most general case. 
 
 \smallskip
 \noindent
 To construct this special case of composition, let us begin by observing that for any symplectic manifold $N$, one has a projection map: $\mathbb{G}(N) \longrightarrow N$ defined as the composite of the the map to $\mathcal{E}(N)$, followed with the projection to $N$. 
 
 \smallskip
 \noindent
 Now define a parametrized $\Omega$-module spectrum defined via the pullback: 
\[
\xymatrix{
\mathbb{G}(\Delta)       \ar[d]^{\xi} \ar[r] &  \mathbb{G}(M \times \overline{M}) \ar[d] \\
M     \ar[r]^{ \Delta \quad} &  M \times \overline{M}
}
\]
where $\Delta$ denotes the diagonal inclusion of $M$ inside $M \times \overline{M}$. 

\medskip
\noindent
At this point, let us observe a few relevant facts:

\begin{enumerate}
\item The structure group for the bundle of $\mathbb{L}$-spectra $\mathbb{G}(M \times \overline{M}) \longrightarrow M \times \overline{M}$ is the compact Lie group $\U(2m)$. As a consequence, the bundle $\mathbb{G}(\Delta)$ enjoys the same property. 
\item Since the symplectic manifold $M$ is endowed with a compatible complex structure, one has an induced Hermitian metric on the tangent bundle of $M$. In particular, the map $\Delta$ admits a neighborhood which can be identified via the exponential map, with a neighborhood of the zero section of the normal bundle of $\Delta$. 
\item The map $\Delta$ is proper, even if the symplectic manifold $M$ is non-compact. 

\end{enumerate}

\medskip
\noindent
Let $\eta(\Delta)$ denote the normal bundle of $\Delta$, which may be canonically identified with the tangent bundle of $M$, denoted by $\tau$. Let $\iota : M \longrightarrow \R_+$ denote a function that is bounded by the injectivity radius of the exponential map: $\mbox{Exp} : \eta(\Delta) \longrightarrow M \times \overline{M}$. In other words, the exponential map sends all vectors of radius bounded by $\iota$ homeomorphically to an open neighborhood of $\Delta$. For a fixed choice of $\iota$, one may define the Pontrjagin--Thom construction along the top horizontal map in the pullback diagram:
\[ \Delta_!(\iota) : \Omega(M \times \overline{M}) = \pi^{M \times \overline{M}}_!(\mathbb{G}(M \times \overline{M})) \longrightarrow \pi^M_! (\Sigma^\tau \mathbb{G}(\Delta)), \]
where, $\Sigma^\tau \mathbb{G}(\Delta)$ denotes the parametrized spectrum over $M$ obtained by suspending $\mathbb{G}(\Delta)$ fiberwise with the fiberwise compactification of the bundle $\tau$. The next step is to prove: 
\begin{lemma} \label{special} The parametrized spectrum $\Sigma^\tau \mathbb{G}(\Delta)$ admits a natural map to the trivial spectrum over $M$, with fiber $\Omega$. In particular there exists a natural map of $\Omega$-module spectra obtained by collapsing $M$ to a point: $\pi :  \pi^M_! (\Sigma^\tau \mathbb{G}(\Delta)) \longrightarrow \pi^M_!(M \times \Omega) \longrightarrow \Omega$. 
As a consequence, we obtain a map of $\Omega$-module spectra $\Omega(M \times \overline{M}) \longrightarrow \Omega$. 
\end{lemma}
\begin{proof}
By the defining pullback diagram, the spectrum $\Sigma^\tau \mathbb{G}(\Delta)$ is described as:
\[ \mathcal{S}(\Delta) \wedge_{\mathcal{L}} \Omega, \]
where $\mathcal{S} (\Delta) = \mathcal{L}^\C(1) \ltimes_{\mathcal{L}_{2m}(1) \times \Or(2m)} (\mathcal{F}(\Delta) \ltimes \mbox{S})$ , with $\mathcal{F}(\Delta)$ being the restriction of the principal bundle $\mathcal{F}(M \times \overline{M})$ along $\Delta$, with structure monoid canonically reduced to $\mathcal{L}_{2m}(1) \times \Or(2m)$. Now by construction, there is a canonical map from $\mathcal{F}(\Delta)$ to $\mathcal{L}(1)$. Using the $\mathcal{L}^\C(1)$-action on $\mbox{S}$, we have the action map:
\[ \mathcal{L}^\C(1) \ltimes_{\mathcal{L}_{2m}(1) \times \Or(2m)} (\mathcal{F}(\Delta) \ltimes \mbox{S}) \longrightarrow \mathcal{L}^\C(1) \ltimes_{\mathcal{L}(1)} \mbox{S} \longrightarrow \mbox{S}. \]
The proof is complete on smashing with $\Omega$ over $\mathcal{L}$.
\end{proof}

\smallskip
\noindent
Now let us fix an orthogonal isomorphism $\gamma_0 : \R^\infty \times \R^\infty \longrightarrow \R^\infty$ defined as:
\[ \gamma_0(e_j) = e_{2j}, \quad \gamma_0(f_j) = e_{2j+1}, \]
where $e_j,f_j$ represent the standard basis of $\R^\infty \times \R^\infty$. Using $\gamma_0$, get an internal product map:
\[ {\gamma_0}_\ast : \Omega(M) \wedge_\Omega \Omega(M) \longrightarrow \Omega(M \times \overline{M}). \]
We therefore have a composite map after projecting to $\Omega$ in the middle factor:
\[  \Omega(\ast, M) \wedge_\Omega \Omega(M, \ast) \longrightarrow \Omega(\ast) \wedge_\Omega \Omega(M \times \overline{M}) \wedge_\Omega \Omega(\ast)  \longrightarrow \Omega(\ast, \ast). \]
Notice that the only choice we made was a function $\iota : M \longrightarrow \R_+$ bounding the injectivity radius. We therefore get the composition we seek parametrized over all such choices:
\[ \mathcal{O} :  \mathcal{O}(M)_+ \wedge \Omega(\ast, M) \wedge_\Omega \Omega(M,\ast)  \longrightarrow \Omega(\ast, \ast), \]
where $\mathcal{O}(M)$ is the contractible space of maps $\iota : M \longrightarrow \R_+$ that bound the injectivity radius of the exponential map: $\mbox{Exp} : \eta(\Delta) \longrightarrow M \times \overline{M}$.

\medskip
\noindent
The general case of composition is a simple generalization of the above special case.

\medskip
\noindent
{\bf The general case of composition:}

\medskip
\noindent
We are now ready to define composition in complete generality. Consider $k+2$ objects objects $M_i$ with $0 \leq i \leq k+1$. We seek to define a morphism of $\Omega$-modules:
\[ \mathcal{O}_k : \mathcal{O}(M_1,\ldots,M_k)_+ \wedge \Omega(M_0,M_1) \wedge_\Omega \cdots \wedge_\Omega \Omega(M_k, M_{k+1}) \longrightarrow  \Omega(M_0,M_{k+1}). \]
The space $\mathcal{O}(M_1,\ldots,M_k)$ above will simply be defined as the product:
\[ \mathcal{O}(M_1,\ldots,M_k) = \mathcal{O}(M_1) \times \cdots \times \mathcal{O}(M_k), \]
and the composition will simply be a global version of the previous special case. In detail, we proceed as follows:  First recall that by definition, each morphism $\Omega_{\zeta}(M_i, M_{i+1})$ is written as $\Omega(\overline{M}_i) \wedge_{\Omega} \Omega(M_{i+1})$. 

\smallskip
\noindent
Next we regroup the smash product of the terms $\Omega_{\zeta_{i+1}}(M_i, M_{i+1})$ to get the product:
\[  \Omega(\overline{M}_0) \wedge_{\Omega} \Omega(M_1) \wedge_\Omega \Omega(\overline{M}_1) \wedge_\Omega \cdots \wedge_\Omega \Omega(M_k) \wedge_\Omega \Omega(\overline{M_k}) \wedge_{\Omega} \Omega(M_{k+1}). \]

\smallskip
\noindent
Finally we perform the composition $\mathcal{O}$ to obtain a projection $\Omega(M_i) \wedge_{\Omega} \Omega(\overline{M}_i) \longrightarrow \Omega$, for each $0 < i < k+1$ (parametrized by the product of spaces $\mathcal{O}(M_i)$). We thus obtain a map to $\Omega(\overline{M}_0) \wedge_\Omega \Omega(M_k)$ which we identify with $\Omega(M_0,M_k)$ by definition. The composite of the steps above gives us the general composition map $\mathcal{O}_k$:
\[ \mathcal{O}_k : \mathcal{O}(M_1,\ldots,M_k)_+ \wedge \Omega(M_0, M_1) \wedge_\Omega \cdots \wedge_\Omega \Omega(M_k, M_{k+1}) \longrightarrow \Omega(M_0, M_{k+1}). \]
It is straightforward to verify that $\mathcal{O}_k$ satisfy associativity.

\begin{remark} \label{alt}
Notice that the spaces $\mathcal{O}(X)$ parametrize the space of $A_\infty$-structures. In other words, if one were to fix a choice of element $\iota(X) \in \mathcal{O}(X)$ for each object $X$ in $\mathbb{S}$, once and for all, then for these choices we would get a particular $A_\infty$-structure on $\mathbb{S}$. Alternatively, one may define the objects of $\mathbb{S}$ to include these choices. 
\end{remark}

\medskip
\noindent
{\bf An $A_\infty$-neighborhood of the identity:}

\medskip
\noindent
The next piece of structure we need to explore is the existence of the identity element of an object $(M,\omega)$ in $\mathbb{S}$. Such an element is a distinguished point in $\Omega^{\infty}_{\zeta}(M, M)$, or equivalently, a stable map $[id] : \mbox{S} \longrightarrow \Omega(M,M)$. This map would need to satisfy some obvious identities that are required of an identity in a category. Recall that in $h\mathbb{S}$, such an element did indeed exist if $(M,\omega)$ was a compact symplectic manifold. In the $A_\infty$-category $\mathbb{S}$, it appears that an honest identity element is too much to ask for (see remark \ref{unit} below). This issue is not new, see \cite{C} for example. Instead, for arbitrary manifolds $(M,\omega)$ we will construct a map of (possibly unit less) $A_\infty$-ring spectra:
\[ \Delta^{-\tau} : M^{-\tau} \longrightarrow \Omega(M,M), \]
where $M^{-\tau}$ is a suitable model for the Thom spectrum of the stable normal bundle of $M$. We will interpret the map $\Delta^{-\tau}$ as a neighborhood of the identity since it has the property that the identity morphism $[id]$ factors through it in $h\mathbb{S}$, when $M$ is compact. 

\begin{remark} \label{unit}
We have chosen to parametrize our spectra $\mathbb{G}(M)$ over the space $\mathcal{E}(M)$ of all unitary injections $TM \longrightarrow M \times \C^\infty$. This is the reason one fails to have a strict unit for $\Omega(M,M)$. We could have chosen to fix a choice of such an injection for each manifold $M$ to begin with, and defined the objects $\mathbb{G}(M)$ to be parametrized over $M$, for this given injection. One would then get an $A_\infty$-cateogry $\mathbb{S}$ as before that we believe admits strict identity elements (for compact $M$), using an argument similar to \cite{C}. However, for aesthetic reasons, we have avoided making such choices. 
\end{remark}

\begin{remark} \label{dq}
The spectrum $M^{-\tau}$ is in fact an $E_\infty$-ring spectrum. This is equivalent to observing that the product is symmetric under the twist map. However, the spectrum $\Omega(M,M)$ is certainly not $E_\infty$. In section \ref{stab}, we suggested an interpretation of the $A_\infty$-neighborhood of the identity as a homotopical analog of the fact that the algebra of observables on a symplectic manifold $M$ is a deformation of the commutative ring $\mbox{C}^\infty(M)$. 
\end{remark}

\medskip

\begin{defn}
Consider the isomorphism given by the isometry: $\gamma_0 : \C^\infty = \R^\infty \times i\R^\infty \longrightarrow \R^\infty$ that identifies the standard basis $(e_j,ie_j)$ by $\gamma_0(e_j) = e_{2j}$ and $\gamma_0(ie_j) = e_{2j+1}$. 

\smallskip
\noindent
Let $\mbox{S}^{-2m} = \mbox{S}^{-\C^{m}}$ denote the desuspended sphere spectrum. Then the model for $M^{-\tau}$ we use is defined as the $\mbox{S}$-module: $\pi^M_!(\mathbb{T}_\tau(M))$, where $\mathbb{T}_\tau(M)$ is the $\mbox{S}$-module parametrized over $\mathcal{E}(M)$:
\[  \mathbb{T}_\tau(M) = \mathcal{S}_\tau(M) \wedge_{\mathcal{L}} \mbox{S}, \]
and where $\mathcal{S}_\tau(M) = \mathcal{L}^\C(1) \ltimes^{\gamma_0}_{\mathcal{L}_m^\C(1) \times \U(m)} (\mathcal{F}(M) \ltimes \mbox{S}^{-2m})$. Here $\mathcal{L}^\C(1) \ltimes^{\gamma_0}_{\mathcal{L}_m^\C(1) \times \U(m)}$ indicates the action of $\mathcal{L}_m^\C(1) \times \U(m)$ on $\mathcal{L}^\C(1)$ through the sub monoid $\gamma_0(\mathcal{L}_m^\C(1) \times \U(m)) \subseteq \mathcal{L}(1)$. 
\end{defn}

\medskip
\noindent
Now the isometry $\gamma_0$ induces a map of parametrized spectra: 
\[ \Delta_\ast : \mathcal{S}_\tau(M) \longrightarrow \mathcal{S}(\overline{M}) \wedge_{\mathcal{L}} \mathcal{S}(M) \]
over the diagonal map: $\Delta : \mathcal{E}(M) \longrightarrow \mathcal{E}(\overline{M}) \times \mathcal{E}(M)$. Extending by the unit map $\mbox{S} \longrightarrow \Omega$, we get the map parametrized over $\Delta$:
\[ \Delta_\ast : \mathbb{T}_\tau(M) \longrightarrow \mathbb{G}(\overline{M}) \wedge_\Omega \mathbb{G}(M). \]
Applying $\pi^{\overline{M} \times M}_!$ to the above map, yeilds 
\[ \Delta^{-\tau} : M^{-\tau} \longrightarrow \Omega(M,M). \]
It remains to show that this is a map of $A_\infty$-ring spectra. In the process of doing so, we shall also describe the $A_\infty$-structure on $M^{-\tau}$. 

\smallskip
\noindent
Recall that composition $\mathcal{O}_n$ was defined as the Pontrjagin--Thom consruction performed after the internal product map:
\[ \Omega(M,M) \wedge_\Omega \cdots \wedge_\Omega \Omega(M, M) \longrightarrow \Omega(\overline{M}) \wedge_\Omega \Omega(M \times \overline{M})^{\wedge n} \wedge_\Omega \Omega(M). \]
Notice that the (totally geodesic) sub manifold $\Delta(M)^{\times (n+1)} \subset (\overline{M} \times M)^{\times (n+1)}$ intersects the manifold $\overline{M} \times \Delta(M)^{\times n} \times M$ transversally along the thin diagonal $\Delta_{(n+1)}(M) \subset \Delta(M)^{\times (n+1)}$. In particular, each $n$-touple of functions $(\iota_1, \ldots, \iota_n)$ on $\Delta(M)^{\times n}$ that when extended trivially to $\overline{M} \times \Delta(M)^{\times n} \times M$, are bounded by the injectivity radius of the inclusion 
\[ \overline{M} \times \Delta(M)^{\times n} \times M \longrightarrow (\overline{M} \times M)^{\times (n+1)} = \overline{M} \times (M \times \overline{M})^{\times n} \times M \]
restricts to a function $\Delta(\iota)$ on $\Delta_{(n+1)}(M)$ bounded by the injectivity radius of the inclusion $\Delta_{(n+1)}(M) \subset \Delta(M)^{\times (n+1)}$. In addition, notice that the normal bundle of the inclusion $\Delta_{(n+1)}(M) \subset \Delta(M)^{\times (n+1)}$ is canonically isomorphic to $\tau^n$. Hence, performing the Pontrjagin--Thom construction along the transverse intersection, we get the commutative diagram parametrized over all $n$-touples of functions $(\iota_1, \ldots, \iota_n)$ that satisfy the injectivity property for the exponential map: 

\[
\xymatrix{
 \mathcal{O}(M, \cdots, M)_+ \wedge M^{-\tau} \wedge_{\mathcal{L}} \cdots \wedge_{\mathcal{L}} M^{-\tau} \ar[r]^{\hspace{1.1in} \mathcal{P}_{n}} \ar[d]^{(\Delta^{-\tau})^{\wedge (n+1)}} & M^{-\tau} \ar[d]^{\Delta^{-\tau}} \\
 \mathcal{O}(M, \cdots, M)_+ \wedge \Omega(M,M) \wedge_\Omega \cdots \wedge_\Omega \Omega(M,M)\ar[r]^{\hspace{1.3in} \mathcal{O}_{n}} & 
\Omega(M,M),}
\]
where the top horizontal map $\mathcal{P}_{n}$ is the product of $(n+1)$-objects $M^{-\tau}$ induced via the Pontrjagin--Thom collapse along the thin diagonal: $\Delta_{(n+1)}(M) \subset M^{\times (n+1)}$. 

\medskip
\noindent
If one picks a point $\iota \in \mathcal{O}(M)$, then one gets an induced point $(\iota, \ldots, \iota) \in \mathcal{O}(M,\ldots, M)$. For such choice, we see that $M^{-\tau}$ is an $A_\infty$-ring spectrum via the maps $\mu$ above. In particular, the above diagram demonstrates that the map $\Delta^{-\tau} : M^{-\tau} \longrightarrow \Omega(M,M)$ is a map of $A_\infty$-ring spectra, with the $A_\infty$-structures on either spectrum being parametrized by $\mathcal{O}(M)$.

\begin{remark} \label{dq2}
Smashing the map $\Delta^{-\tau} : M^{-\tau} \longrightarrow \Omega(M,M)$ with $\Omega$, we get a map:
\[ \lambda : \Omega \wedge_{\mathcal{L}} M^{-\tau} \longrightarrow \Omega(M,M). \]
Recall from example \ref{cotang} that $\Omega \wedge_{\mathcal{L}} M^{-\tau}$ is homotopy equivalent to $\Omega(T^\ast M)$. Indeed, the map $\lambda$ is induced by  a symplectic immersion: $T_\epsilon^\ast M \longrightarrow \overline{M} \times M$ about an $\epsilon$-neighborhood of the diagonal. 
\end{remark}

\section{The Stable Metaplectic Category}\label{meta}

\noindent
K\"{a}hler manifolds admit a construction known as holomorphic (or K\"ahler) quantization. The state space is the space of square integrable holomorphic sections of the line bundle given by $\mathcal{L} \otimes \sqrt{\det}$, where $\mathcal{L}$ is the prequantum line bundle, and $\sqrt{\det}$ is a choice of square root of the volume forms (called a metaplectic structure). This construction is not functorial on the symplectic category however. Indeed, it is well known that one cannot expect to construct an honest quantization functor on the symplectic category, with values in the category of topological vector spaces. One may therefore attempt to construct a ``derived version of geometric quantization''. A derived version of this quantization should in principle depend on all the cohomology groups of this line bundle. Now the Dolbeaut $\db$-complex computing the cohomology of holomorphic bundles agrees with the $\Spin^c$ Dirac operator. On incorporating the metaplectic structure into the picture, the complex agrees with the $\Spin$-Dirac operator. We therefore observe that almost complex $\Spin$ manifolds support a Dirac operator which generalizes the Dolbeaut complex twisted by the square root of the volume forms. This may suggest {\em defining} a derived version of the geometric quantization of a symplectic manifold with a $\Spin$ structure as the $L^2$ index of the Dirac operator with coefficients in the prequantum line bundle. For compact manifolds this is simply the $\hat{A}$ genus with values in the prequantum line bundle. 

\smallskip
\noindent
We take the above discussion as good motivation to define a variant of the stable symplectic category with objects supporting this structure, we call this the {\em Stable Metaplectic Category}. The actual application to construction a derived geometric quantization as a functor on the stable metaplectic category will appear in a later document \cite{N2}. A very brief outline of the framework is described at the end of this section. 

\medskip
\noindent
The objects in the metaplectic category will be symplectic manifolds endowed with a compatible metaplectic structure. Informally speaking, a compatible metaplectic structure is a compatible complex structure endowed with a square root of the determinant line bundle. Let us now formalize the concept of a metaplectic structure. Let the classifying space of the metaplectic group $\tU$, denoted by $\BtU$, be defined via the fibration:
\[ \BtU \longrightarrow \BU \longrightarrow \KZtt \] 
with the second map being the mod-2 reduction of the first Chern class. By definition, $\BtU$ supports the square-root of the determinant map $\sqrt{\det} : \BtU \longrightarrow \BSo$. 
The complexification map $\BSO \longrightarrow \BU$ lifts to a unique map $\BSpin \longrightarrow \BtU$. We may describe these lifts as a diagram of fibrations:

\[
\xymatrix{
\tUSpin        \ar[d] \ar[r] & \USO \ar[d] \ar[r] & \KZto \times \KZtt \ar[d] \\
\BSpin      \ar[r] \ar[d] & \BSO \ar[d] \ar[r]^{w_2} & \KZtt \ar[d]^{0} \\
\BtU \ar[r] & \BU \ar[r]^{c_1} & \KZtt 
}
\]

\noindent
An easy calculation shows that the mod-2 cohomology $H^\ast(\USO, \Z/2)$ is an exterior algebra on generators $\{ \sigma, w_2, w_3, \ldots \}$, where $\sigma$ is the class that transgresses to $c_1$ in the mod-2 Serre spectral sequence, and $w_i$ are the classes that are given by the corresponding restrictions from $H^\ast(\BSO, \Z/2)$. In addition, the class $\USO \longrightarrow \KZto \times \KZtt$ is represented by the product $\sigma \times w_2$. Hence $\tUSpin$ is the corresponding cover of $\USO$. 

\begin{remark}
The kernel of the map $\tU \longrightarrow \UO$ is given by a group $\tilde{\Or}$ called the metalinear group (see \cite{W2} Section 7.2). The classifying space $\B$ is described by a fibration:

\[
\xymatrix{\B      \ar[r] \ar[d] & \BO \ar[d] \ar[r]^{w_1^2 \quad } & \KZtt \ar[d]^{=} \\
\BtU \ar[r] & \BU \ar[r]^{c_1 \quad} & \KZtt 
}
\]
The group $\tilde{\Or}$ can easily be seen as a split $\Z/4$ extension of $\SO$. The map $\B \longrightarrow \BZf$ is sometimes called the Maslov line bundle. Notice also that $\BSpin$ can be seen as the cover of $\B$ given by prescribing a trivialization of the Maslov line bundle, followed by a $\Spin$ structure.
\end{remark}
\noindent
The space $\tUSpin$ supports a stable vector bundle that lifts $\zeta$, which we will denote by the same name. The spectrum $(\tUSpin)^{-\zeta}$ is an $E_\infty$-ring spectrum modeled on an operad $\tilde{\mathcal{L}}(k)$ which is defined as the contractible space of spin isometries between ${(\R^\infty)}^{\times k}$ and $\R^\infty$. The hermitian isometries are replaced by the Metaplectic isometries $\tilde{\mathcal{L}}^\C(k)$. The rest of the theory goes through verbatim as before, and one defines a ``coefficient" $\mbox{S}$-algebra: $\tilde{\Omega} = \mbox{S} \wedge_{\tilde{\mathcal{L}}} (\tUSpin)^{-\zeta}$. One can now define the stable metaplectic category in analogy with the stable symplectic category: The objects of the stable metaplectic category $\tilde{\mathbb{S}}$ are symplectic manifolds $(M,\omega)$ endowed with a compatible complex structure. As part of the data, we also fix a metaplectic structure: i.e. a lift of the structure group of the tangent bundle to the double cover $\tilde{\U}(m)$ of $\U(m)$ given by the restriction of the double cover $\tilde{\U} \longrightarrow \U$, along $\U(m)$. In particular, $\tilde{\U}(m)$ corresponds to a compatible family of square-roots of the determinant homomorphism. Given a $2m$ dimensional metaplectic manifold $M$, we may define $\tilde{\mathcal{E}}(M)$ to be the space of $\tilde{\U}(m)$-lifts of unitary maps from $TM$ to $M \times \C^\infty$.

\begin{defn}
Let $\tilde{\mathbb{G}}(M)$ denote the spectrum parametrized over the space $\tilde{\mathcal{E}}(M)$ given by $\tilde{\mathcal{S}}(M) \wedge_{\tilde{\mathcal{L}}} \tilde{\Omega}$, where $\tilde{\mathcal{S}}(M)$ denotes the fiberwise $\tilde{\mathbb{L}}$-spectrum:
\[ \tilde{\mathcal{S}}(M) = \tilde{\mathcal{L}}^\C(1) \ltimes_{\tilde{\mathcal{L}}_m^\C(1) \times \tilde{\U}(m)} (\tilde{\mathcal{F}}(M) \ltimes \mbox{S}^{-m}). \] 
Then we define the morphism spectra $\tilde{\Omega}(M,N)$ in $\tilde{\mathbb{S}}$ as:
\[ \tilde{\Omega}(M,N) := \tilde{\Omega}(\overline{M}) \wedge_{\tilde{\Omega}} \tilde{\Omega}(N) =   \pi^{\overline{M}}_! \tilde{\mathbb{G}}(\overline{M}) \wedge_{\tilde{\Omega}} \pi^{N}_! \tilde{\mathbb{G}}(N).  \]
Composition is defined analogously to the stable symplectic category, and as before, all structure maps in the category are module maps over the coefficient spectrum $\tilde{\Omega}$.
\end{defn}

\noindent
As before, one has a stable metaplectic homotopy category that captures the geometry. In analogy with remark \ref{hpty}, the geometric object underlying the spectrum $\pi^M_!(\tilde{\mathbb{G}}(M))$ is the Thom spectrum $\tilde{\Omega}(M)= \tilde{\underline{\mathcal{G}}}(M)^{-\zeta}$, with $\tilde{\underline{\mathcal{G}}}(M)$ defined as the pullback:
\[
\xymatrix{
\tilde{\underline{\mathcal{G}}}(M)    \ar[d]^{\xi} \ar[r]^{\zeta \quad } & \Z \times \BSpin \ar[d] \\
M      \ar[r]^{\tau \quad} & \Z \times \BtU
}
\]
where $\tau$ denotes the $m$-dimensional tangent bundle of the metaplectic manifold $M$. Points in the infinite loop space $\Omega^\infty(\Omega(M,N))$ represent totally real (or lagrangian) immersions into $\overline{M} \times N$ of manifolds with a spin structure. Furthermore, the morphism spectra in the metaplectic category admit maps to $\MSpin$ that are canonical up to homotopy. 

\medskip
\noindent
In \cite{N2}, we hope to give a geometric description of the functor from a category closely related to the metaplectic category (called the category of symbols), to the $\mbox{KU}$-linear category obtained by ``extending coefficients" over $\Omega$, to complex $\mbox{K}$-theory along the $\hat{A}$-genus. Under this functor, a lagrangian immersion $L \looparrowright M$ will be associated to a Fredholm operator on $M$. This operator will be given represented by the $\mbox{K}$-class of (derived) flat sections of the pre-quantum line bundle restricted to $L$. The induced $A_\infty$-composition can now be seen as a functorial way of composing symbols associated to lagrangian immersions. We would like to interpret this as a derived framework for geometric quantization.

\section{Appendix: Some computations and Remarks}

\noindent
{\bf The Unoriented Case:}

\medskip
\noindent
Let us make some explicit computations in the case of the unoriented symplectic category $\mathbb{S}$. We invoke the Adams spectral sequence to compute $\pi_\ast \Omega(M,N)$. Since $\Omega(M,N)$ is a generalized Eilenberg--Mac\,Lane spectrum, the spectral sequence will collapse and we simply need to compute the primitives under the action of the dual mod-2 Steenrod algebra on $H_\ast (\Omega(M,N), \Z/2)$. 

\begin{conv}
Let us set some notation. All homology groups will be understood to be over $\Z/2$. In addition, given a real vector bundle $\zeta$ of dimension $k$, let us use the suggestive notation $\Sigma^{-\zeta} S_\ast$ to denote the shift $\Sigma^{-k} S_\ast$ for a graded module $S_\ast$. 
\end{conv}

\begin{thm}
$\pi_\ast \Omega(M,N)$ is a free $\pi_\ast \Omega$-module on a (non-canonical) generating vector space given by $\Sigma^{-\zeta}H_\ast(\overline{M} \times N)$.  
\end{thm}
\begin{proof}
The Thom isomorphism theorem implies that the ring $H_\ast \Omega$ is isomorphic to $H_\ast(\UO)$. Under this isomorphism we also see that $H_\ast \Omega(M,N)$ is isomorphic to  $\Sigma^{-\zeta} H_\ast (\underline{\mathcal{G}}(\overline{M} \times N))$. Now consider the universal fibration $\UO \rightarrow \BO \longrightarrow \BU$. It is easy to see that the Serre spectral sequence in homology for this fibration collapses leading to the fact that $H_\ast (\underline{\mathcal{G}}(\overline{M} \times N))$ is non-canonically a free $H_\ast(\UO)$-module on $H_\ast(\overline{M} \times N)$. From this we deduce that $H_\ast  \Omega(M,N)$ is free $H_\ast \Omega$-module on the (non-canonical) vector space given by $\Sigma^{-\zeta} H_\ast(\overline{M} \times N)$. An easy argument using the degree filtration shows that the generating vector space can be chosen to have trivial action of the dual Steenrod algebra. The statement of the theorem is now complete on taking primitives under this action. 
\end{proof}
\noindent
The following consequences of the above theorem are easy (compare with \ref{Weyl1}): 
\begin{thm}
There is a natural decomposition of $\pi_\ast \Omega(M,N)$ induced by the composition map:
\[ \pi_\ast \Omega(M,\ast) \otimes_{\pi_\ast \Omega} \pi_\ast \Omega(\ast, N) = \pi_\ast \Omega(M,N). \]
In particular, arbitrary compositions can be canonically factored in homotopy, and computed by applying the composition map internally: 
\[ \pi_\ast \Omega(\ast, N) \otimes_{\pi_\ast \Omega} \pi_\ast \Omega(N, \ast) \longrightarrow \pi_\ast \Omega.
\]
\end{thm}
\begin{thm} \label{Weyl}
Given a compact manifold $M$, the $\pi_\ast \Omega$-algebra $\pi_\ast \Omega(M,M)$ has the structure of an endomorphism algebra:
\[ \pi_\ast \Omega(M,M) = \End_{\pi_\ast \Omega} \left( \pi_\ast \Omega(\ast, M) \right). 
\]
\end{thm}

\noindent
{\bf The Oriented Monotone Case:}

\smallskip
\noindent
Next, let us very briefly explore the structure of $s\Omega(M)$ rationally in the case when $M$ is monotone, i.e. when the cohomology class of $\omega$ is a non-zero multiple of the first Chern class of $M$. 

\smallskip
\noindent
Firstly recall that $H^\ast(\USO,\Q)$ is an exterior algebra $\Lambda(y_{4i+1})$. Now by Thom isomorphism, we have an equality $H^{\ast+m}(s\underline{\mathcal{G}}(M), \Q) = H^\ast(s\Omega(M), \Q)$, where $(M,\omega)$ is a $2m$-dimensional manifold. Now consider the cohomology Serre spectral sequence for the fibration 
\[ \USO \longrightarrow s\underline{\mathcal{G}}(M) \longrightarrow M. \]
It is easy to see, using the monotonicity assumption, that the class $y_1 \in H^1(\USO, \Q)$ transgresses to a non-trivial multiple of the symplectic class $\omega$. Hence the class $y_1 \cup \omega^m$ (uniquely) represents a class in $H^{2m+1}(s\underline{\mathcal{G}}(M), \Q)$, that is the only meaningful primary characteristic class. Let $\theta(M)$ be the corresponding class in $H^{m+1}(s\Omega(M), \Q)$ under the above Thom isomorphism. 

\medskip
\noindent
Now let $\pi : E \longrightarrow B$ be a fibrating family of oriented stable lagrangians in $M$, endowed with a classifying map $f(\pi) : B \longrightarrow \Omega^{\infty} (s\Omega(M))$.  Then the map $f(\pi)$ factors through the Umkehr map $B_+ \longrightarrow \Omega^{\infty} (E^{-\tau(\pi)})$ followed by the map induced by $E^{-\tau(\pi)} \longrightarrow s\Omega(M)$. It follows that 
\[ f(\pi)^\ast \theta(M) = \pi_\ast (y_1 \cup \omega^m), \]
where $y_1 \cup \omega^m$ denotes the pullback of the class having the same name along $E \longrightarrow s\underline{\mathcal{G}}(M)$.

\bigskip
\noindent
{\bf Some questions and remarks:}

\medskip
\noindent
Here is a list of natural questions and relevant remarks:  

\begin{question}
Is there a universal description of the stable symplectic or metaplectic category that allows us to check if a functor defined on symplectic manifolds extends to the stable category? 
Notice that if $\Fm(M)$ denotes any (stable) representation of $\mathbb{S}$, with $\Fm(\ast) := \Fm$ being an $\Omega$-module, then we have the action map for $\Fm$ on the level of spectra:
\[ q : \Omega(M, \ast) \wedge_{\Omega} \Fm(M) \longrightarrow \Fm. \] 
For compact manifolds $M$, we may dualize this map to get a natural transformation:
\[ \Fm(M) \longrightarrow \Omega (\ast, M) \wedge_{\Omega} \Fm. \]
Hence functors of the form $\Omega(\ast,M) \wedge_\Omega \Fm$ are terminal in the category of all functors defined on the (subcategory of compact objects in the) stable symplectic category.  
\end{question}

\begin{question} \label{GT}
One would like to describe the ``Motivic Galois group", by which we mean the rule that assigns to a commutative $\Omega$-algebra $\Fm$, its $\Fm$-points given by the group of multiplicative automorphisms of the monoidal functor on $\mathbb{S}$ (with values in the category of $\Fm$-modules), and which takes a symplectic manifold $N$ to $\Omega(\ast, N) \wedge_{\Omega} \Fm$. The author and J. Morava have studied this Motivic Galois group. We show in \cite{KM} that this group contains a natural subgroup which can be identified (over $\Q$), with a graded vector space with generators in degree $4k+2$, for $k \geq 0$. There is a striking similarity between this subgroup, and the abelian quotient of the Grothendieck-Teichm\"uller group that is known to act on deformation quantization \cite{KO}.

\end{question}

\begin{question}
Recall the homomorphism from the symplectomorphism group to the units $\GL s\Omega(M,M)$ given by taking a symplectomorphism to its graph:
\[ \gamma: \BSymp(M,\omega) \longrightarrow \BGL  s\Omega(M,M). \]
One may map $\BGL  s\Omega(M,M)$ into the Waldhausen K-theory of $s\Omega$, denoted by $\mbox{K}(s\Omega)$, and ask for this invariants of $\BSymp(M,\omega)$ with values in $\mbox{K}(s\Omega)$. In \cite{N,KM}, we give a complete description of this map rationally. In particular, we observe that $\pi_\ast(\mbox{K}(s\Omega))$ is a direct sum of a polynomial algebra on classes in degree $4n+2$, for $n \geq 0$ (related closely to the Motivic Galois group), and the graded vector space $\pi_\ast (\mbox{K}(\Z)) \otimes \Q$ which is detected by the zero section: $s\Omega \longrightarrow \mbox{H}(\Z)$. It would be very interesting to identify the geometry underlying the generators of $\pi_\ast(\mbox{K}(\Z)) \otimes \Q$. In particular one would like to know if the generator of $\pi_\ast(\mbox{K}(\Z)) \otimes \Q$ in degree $4n+1$ is related to $\pi_{4n} \BSymp_c(\C^k)$, for large $k$, where $\Symp_c(\C^k)$ denotes compactly supported symplectomorphisms?
\end{question}

\newpage

\pagestyle{empty}
\bibliographystyle{amsplain}
\providecommand{\bysame}{\leavevmode\hbox
to3em{\hrulefill}\thinspace}

\end{document}